\documentclass{siamltex}

\usepackage{amsmath,amssymb}

\usepackage[english]{babel}
\usepackage[T1]{fontenc}
\usepackage{latexsym}
\usepackage{subfigure, epsfig, epsf}
\usepackage{color,fancyhdr,lastpage,graphicx}
\usepackage{xspace}
\usepackage{setspace}
\usepackage{bbm}
\usepackage{wasysym}
\usepackage[lined,boxed]{algorithm2e}
\usepackage{stmaryrd}

\newcommand{\RR}{\mathbb{R}}    \newcommand{\LL}{\mathbb{L}}

   \newcommand{\mcU}{\mathcal{U}}
  \newcommand{\mcM}{\mathcal{M}} 
  \newcommand{\mcN}{\mathcal{N}} \newcommand{\mcD}{\mathcal{D}}

 \newcommand{\la}{\leftarrow}

\newcommand{\EE}{\mathbb{E}} \newcommand{\PP}{\mathbb{P}}

\newcommand{\as}{almost surely }

\newcommand{\Var}{\textsc{V}\textrm{ar}}

\newcommand{\st}{such that }

\renewcommand{\leq}{\leqslant}
\renewcommand{\geq}{\geqslant}

\renewcommand{\la}{\lambda}
\newcommand{\al}{\alpha}

\newcommand{\wrt}{with respect to }
\renewcommand{\st}{such that }
\newcommand{\wlg}{without loss of generality }
\newcommand{\rhs}{right hand side }
\newcommand{\lhs}{left hand side }

\newcommand{\Khat}{\widehat{K}}

\newcommand{\rhohat}{\widehat{\rho}}

\newcommand{\eq}{eq.\ }

\title{A Stochastic Algorithm for Parametric Sensitivity in Smoluchowski's Coagulation Equation}

\author{{Isma\"el~F.~Bailleul\footnotemark[2]\;\,\footnotemark[3]
, Peter~L.~W.~Man\footnotemark[4]\;\,
\and Markus~Kraft\footnotemark[4]}}

\footnotetext[2]{
Department of Pure Mathematics and Mathematical Statistics,
Centre for Mathematical Sciences,
University of Cambridge,
Wilberforce Road,
Cambridge, CB3 0WB,
UK (i.bailleul@statslab.cam.ac.uk)}

\footnotetext[3]{
This research was supported by EPSRC grant EP/E01772X/1.}

\footnotetext[4]{
Department of Chemical Engineering and Biotechnology,
University of Cambridge,
New Museums Site,
Pembroke Street,
Cambridge, CB2 3RA,
UK (mk306@cam.ac.uk)}

\begin{document}


\maketitle

\selectlanguage{english}
\begin{abstract}
In this article a stochastic particle system approximation to the parametric sensitivity in the Smoluchowski coagulation equation is introduced. The parametric sensitivity is the derivative of the solution to the equation with respect to some parameter, where the coagulation kernel depends on this parameter. It is proved that the particle system converges weakly to the sensitivity as the number of particles N increases. A Monte Carlo algorithm is developed and variance reduction techniques are applied. Numerical experiments are conducted for two kernels: the additive kernel and one which has been used for studying soot formation in a free molecular regime. It is shown empirically that the techniques for variance reduction are indeed very effective and that the order of convergence is O(1/N). The algorithm is then compared to an algorithm based on a finite difference approximation to the sensitivity and it is found that the variance of the sensitivity estimators are considerably lower than that for the finite difference approach. Furthermore, two methods of establishing `efficiency' are considered and the new algorithm is found to be significantly more efficient.
\end{abstract}

\medskip



\pagestyle{myheadings}
\thispagestyle{plain}
\markboth{I.~F.~BAILLEUL, P.~L.~W.~MAN AND M.~KRAFT}{
A STOCHASTIC ALGORITHM FOR PARAMETRIC SENSITIVITY}

%

\section{Introduction}

Smoluchowski's description of a coagulation process is made in terms of densities $\mu_t(x)$ of particles of mass $x=1,2,3,\ldots$ and takes the form of an infinite dimensional differential equation
\begin{equation}
\label{SmoDiscrete}
\frac{d}{dt}\mu_t(x) = \frac{1}{2}\sum_{y=1}^{x-1}\, K(y,x-y)\,\mu_t(y)\,\mu_t(x-y) - \mu_t(x)\sum_{y=1}^\infty K(x,y)\,\mu_t(y).
\end{equation}
The symmetric kernel $K(x,y)$ appearing in this equation should be understood as giving the rate at which two particles of mass $x$ and $y$ coagulate. One gets an equivalent and more symmetric equation considering $\mu_t(\cdot)$ as a measure on the set of non-negative integers and looking at the time evolution of observables of the form $(f,\mu_t):=\sum_x f(x)\mu_t(x)$; moments are examples of such observables. In these terms, equation \eqref{SmoDiscrete} takes the form
\begin{equation}
\label{SmoDiscreteSymm}
(f,\mu_t) = (f,\mu_0) + \frac{1}{2}\int_0^t\left(\sum_{x,y\geq 1} \bigl\{f(x+y)-f(x)-f(y)\bigr\}\,K(x,y)\,\mu_s(x)\,\mu_s(y)\right)\textrm{d}s.
\end{equation}
The basic problem we address is to derive a numerical scheme to understand how the solution to this equation depends on possible parameters in the kernel. We shall write $K_\la$ to indicate that $K$ depends on some $d$-dimensional parameter $\la$, and shall write $\mu_t^\la$ for the solution of equation \eqref{SmoDiscreteSymm}. Formally differentiating this equation \wrt $\la$ and setting $\sigma_t^\la = \partial_\la\mu_t^\la$ we get
\begin{equation}
\label{SensitivityEq}
\begin{split}
(f,\sigma_t^\la) = (f,\sigma_0^\la) &+ \frac{1}{2}\int_0^t \Bigl(\sum_{x,y\geq 1}\bigl\{f(x+y)-f(x)-f(y)\bigr\}\,K_\la(x,y)\,\mu^\la_s(x)\,\sigma^\la_s(y)\Bigr)\,\textrm{d}s \\
&+ \int_0^t \Bigl(\sum_{x,y\geq 1}\bigl\{f(x+y)-f(x)-f(y)\bigr\}\,K'_\la(x,y)\,\mu^\la_s(x)\,\mu^\la_s(y)\Bigr)\,\textrm{d}s.
\end{split}
\end{equation}
$K'_\la$ is here the derivative of $K_\la$ \wrt $\la$. Section \ref{SectionTheoretic} presents an algorithm which simulates the sensitivity $\sigma_t^\la$ very accurately and in an efficient way.

\smallskip
\newpage

There are two main motivations for performing sensitivity analysis. The first is for solving inverse problems. If some particle system is governed by a partial differential equation which in turn is dependent upon some unknown parameter, it is desirable to find this parameter. This can be achieved by choosing the parameter value which minimises some residual which is a function of experimentally realised quantities and its computational analogue, which varies with the parameter. The minimisation procedure often uses a gradient search, thus the value of computing parametric derivatives is apparent. Secondly, in considering a scientific model, we often wish to consider the smallest model which reasonably fits the data, in which case sensitivity analysis can be performed to discard parameters with small sensitivity.

\smallskip

Whilst the usual tools of solving differential equations (and their associated numerical schemes) are badly adapted to the above infinite dimensional framework, the stochastic approach of interaction particle systems (basically Markov chains) can be used efficiently, in this setting, as Marcus in \cite{Marcus}, and later Lushnikov in \cite{Lushnikov}, first realised. We follow their approach and give a stochastic particle approximation of the sensitivity $\sigma^\la_t$.

Before running any simulation, one should investigate the well-posedness of equation \eqref{SensitivityEq}: if it had more than one solution it would be unclear what solution a numerical scheme approximates. The most general answers to this theoretical question for Smoluchowski equation were given by Jeon in \cite{Jeon} and Norris in \cite{Nor99}, under a growth assumption on the interaction kernel and a moment condition on the initial condition $\mu_0$. Surprisingly enough, the existence and uniqueness problem for the sensitivity was only solved recently, by Bailleul \cite{BailleulSensitivity}, using methods developed by Kolokoltsov \cite{Kolokoltsov2}. The algorithm developed in this article is the numerical counterpart of this theoretical work\footnote{Consult this article for conditions under which existence and uniqueness of a solution to the sensitivity equation \eqref{SensitivityEq} holds.}.

\smallskip

Three approaches to the simulation of the sensitivity by systems of particles have mainly been used up to now. The first uses weighted particles, as illustrated by Vikhansky and Kraft \cite{Vik04}. They approximate the family of solutions $\bigl\{\mu_t^\la\bigr\}_\la$ by Marcus-Lushnikov processes $\sum_{n\geq 0}w_n(t\,;\,\la)\delta_{x_n(t)}$ where the dependence on $\la$ is entirely put on the weights $w_n(t\,;\,\la)$. A heuristic argument imposes to their derivative to satisfy a kind of Markov evolution rule. Despite its (numerically verified) convergence this approach essentially has the same speed of convergence and variance as the Marcus-Lushnikov process. Further, the paper does not any information regarding computation run times.

The second approach considers adjoint sensitivity \cite{Vik06}. A backward partial differential equation is used rather than a forward one, as used in most other methods. The advantage of this method is that sensitivity for any parameter value is immediate once the computation have been done whereas using the forward equation requires explicit calculation for each parameter value. The disadvantage is that one can only calculate the sensitivities for a particular functional of the particle ensemble.

In the third approach, devised by the authors with J.~R. Norris in the forthcoming article \cite{PeterJamesMarkus}, the sensitivity $\sigma_t^\la$ is approximated by the ratio $(\mu_t^{\la+\delta\la\,;\,N}-\mu_t^{\la\,;\,N})/\delta\la$, where $\mu_t^{\la+\delta\la\,;\,N}$ and $\mu_t^{\la\,;\,N}$ are two Markus-Lushnikov processes corresponding to close parameters, coupled so as to minimise the difference of their random fluctuations around $\mu_t^{\la+\delta\la}$ and $\mu_t^\la$. This approach leads to a massive decrease of variance but does not improve the speed of convergence of the algorithm.

\medskip

The algorithm we propose improves the variance of the sensitivity estimator and requires a much smaller number of particles to converge. This is described in section \ref{SectionTheoretic}. The reader who is not interested in mathematical details can skip sections \ref{SectionChainGenerator} and \ref{SectionCvgceThm} where it is proven that the particle system introduced in section \ref{SectionTheoretic} converges to the sensitivity. Section \ref{SectionAlgo} presents the algorithm we have used to obtain the numerical results of section \ref{SectionNumericalResults}.

\bigskip

\noindent \textbf{Notation.} We shall prove convergence of the particle system in a general setting where masses of particles can take any positive real value. The densities of particles will then be represented by non-negative measures $\mu_t$ and all sums will be replaced by integrals. In this framework we shall write $(f,\mu)$ for $\int f(x)\mu(\textrm{d}x)$ and Smoluchowski's equation \eqref{SmoDiscreteSymm} will be written
$$
(f,\mu_t) = (f,\mu_0) + \frac{1}{2}\int_0^t\int\bigl\{f(x+y)-f(x)-f(y)\bigr\}\,K(x,y)\,\mu_s(\textrm{d}x)\,\mu_s(\textrm{d}y)\,\textrm{d}s.
$$
We shall formally write it as
\begin{equation}
\label{Smo}
\dot\mu_t^\la = \frac{1}{2}K_\la(\mu^\la_t,\mu^\la_t).
\end{equation}
In the same way, we shall write formally equation \eqref{SensitivityEq} for the sensitivity as
\begin{equation}
\label{SensitivityEq1}
\dot{\sigma}_t^{\la} = K_{\la}\left(\mu^{\la}_t,\sigma_t^{\la}\right) + \frac{1}{2}K'_{\la}\left(\mu^{\la}_t,\mu^{\la}_t\right).
\end{equation}
The integral notation is adopted from now on.


\section{Markov chain approximation}
\label{SectionTheoretic}

It is probably fair to say that although the Smoluchowski equation \eqref{SmoDiscreteSymm} is a deterministic evolution equation it should primarily be thought of as  a deterministic large scale picture of a stochastic mesoscopic dynamics. Indeed, Smoluchowski obtained his equation from a representation of the coagulation process using `particles' moving according to Brownian trajectories whose diffusivity depends on their mass and coagulate when they are close to each other. As explained in the article \cite{Chandrasekhar} of Chandrasekhar, section $6$ of chapter III, in a region of space where the coagulating particles are well mixed, one can forget about their spatial location and obtain a mean-field evolution for their mass distribution. This mean-field picture is provided by Smoluchowski equation. Given in its simple form \eqref{SmoDiscrete}, it is not clear at first sight how one should simulate a solution to this infinite dimensional differential system.

The approach developed by Marcus in his seminal paper \cite{Marcus} in a sense comes back to the primary stochastic description of the coagulation phenomenon and relies on the intuitive content of Smoluchowski equation. Two particles of masses $x$ and $y$ coagulate at rate $K(x,y)$ to create a new particle of mass $x+y$: The particles $x$ and $y$ are removed from the system and the particle $x+y$ added. This motivated Marcus, and later Lushnikov, to represent a particle of mass $x$ by a Dirac mass $\delta_x$ at $x$ and to introduce a strong Markov jump process on the space of discrete non-negative measures with the following simple dynamics. Denote by $\mu_0^N = \displaystyle{\frac{1}{N}\sum_i\delta_{x_i}}$ its initial state and by $\mu^N_t$ its state at time $t$. Associate to each pair $1\leq i<j\leq N$ independent exponential random times $T_{ij}$ with parameter $\displaystyle{\frac{K\bigl(x_i,x_j\bigr)}{N}}$ and set
$$
T = \min \{T_{ij}\,;\,1\leq i<j\leq N\}.
$$
The process remains constant on the time interval $[0,T)$, and if $T=T_{pq}$ it has a jump $\displaystyle{\frac{1}{N}\bigl(\delta_{x_p+x_q} - \delta_{x_p}-\delta_{x_q}\bigr)}$ at time $T$. The dynamics then starts afresh. Note that the new measure at time $T$ is still non-negative, and that the above description leads to a mean jump of the process during a time interval $[t,t+\delta t]$ equal to\footnote{$\mu^N_t$ denotes the state of the process at time $t$.}
$$
\delta t \sum_{x,x'}\bigl(\delta_{x+x'}-\delta_x-\delta_{x'}\bigr)\,K(x,y)\,\mu^N_t(x)\,\mu^N_t(x')
$$
up to terms of order $\frac{\delta t}{N}$ and $o(\delta t)$. This property makes it clear that the process converges to a solution of the Smoluchowski equation as $N$ goes to infinity (under proper conditions), a fact which was used for simulation purposes long before it was proved under general conditions in \cite{Nor99}.

Following the heuristic approach of Marcus and Lushnikov, we are going to give in the next section a particle description of the sensitivity equation
\begin{equation}
\label{SensitivityEq2}
\dot{\sigma}_t^{\la} = K_\la\bigl(\mu^{\la}_t,\sigma_t^{\la}\bigr) + \frac{1}{2}K'_\la\bigl(\mu^{\la}_t,\mu^{\la}_t\bigr).
\end{equation}
To that end, introduce the notation $K'_+ := K'\vee 0$ and $K'_- := K'\wedge 0$ (dropping the index $\la$ for it will be fixed), and write, for a signed measure $\sigma$,
$$
\sigma = \sigma {\bf 1}_{\frac{d\sigma}{d|\sigma|}>0} - |\sigma| {\bf 1}_{\frac{d\sigma}{d|\sigma|}<0} =: \sigma^+ - \sigma^-,
$$
Using this notation, re-write equation \eqref{SensitivityEq2} as
\begin{equation}
\label{CoupledSigma}
\dot\sigma^+_t - \dot\sigma^-_t = \Bigl(K_\la\bigl(\mu^{\la}_t,\sigma_t^+\bigr) + \frac{1}{2}K'_+\bigl(\mu^{\la}_t,\mu^{\la}_t\bigr)\Bigr) - \Bigl(K_\la\bigl(\mu^{\la}_t,\sigma_t^-\bigr) + \frac{1}{2}K'_-\bigl(\mu^{\la}_t,\mu^{\la}_t\bigr)\Bigr)
\end{equation}
This equation will motivate the introduction of the Markov chain described in the next section.

\smallskip

\noindent \textbf{Notation.} Given three non-negative measures $\mu,\sigma^+,\sigma^-$ on $(0,\infty)$ we shall adopt the notation $\mu\oplus\sigma^+\oplus\sigma^-$ to denote the $\RR_+^3$-valued measure on $(0,\infty)^3$. It will clarify the notation to denote by $x\oplus y\oplus z$ the point of $\RR^3$ with co-ordinates $x,y$ and $z$. Given non-negative functions $f,g,h$ on $(0,\infty)$ set
$$
\bigl(f\oplus g\oplus h, \mu\oplus\sigma^+\oplus\sigma^-\bigr) := (f,\mu)\oplus (g,\sigma^+)\oplus (h,\sigma^-).
$$
As we shall simulate both $\mu_t$ and $(\sigma_t^+,\sigma_t^-)$ at the same time, our approximating Markov chain will take values in the set
\begin{equation*}
\mcN :=\bigl\{\mu\oplus\sigma^+\oplus\sigma^-\,;\,\mu,\sigma^+,\sigma^- \textrm{ non-negative discrete measures on }(0,\infty)\bigr\}.
\end{equation*}

\subsection{Chain, generator}
\label{SectionChainGenerator}

In the same way as the right hand side of Smoluchowski equation \eqref{Smo} can be interpreted as the coagulation of particles of $\mu_t$ of mass $x$ and $y$ at rate $K(x,y)$, we are going to follow what equation \eqref{CoupledSigma} suggest and interpret the term $K(\mu_t,\sigma^+_t)$ appearing there as the coagulation of a particle in $\mu_t$ of mass $x$ with a particle in $\sigma^+_t$ of mass $y$ at rate $K(x,y)$. Note that this leads to a jump $\delta_{x+y}-\delta_x-\delta_y$ of $\sigma^+$ which could transform the non-negative measure $\sigma^+_t$ into a signed measure, as the term $\delta_x$ does not necessarily appear inside $\sigma^+_t$ (while $\delta_y$ does). We shall take care of this by adding $\delta_x$ to the negative part $\sigma^-_t$ of $\sigma_t$ instead of subtracting it from $\sigma^+_t$; as we are only interested in the difference $\sigma^+_t - \sigma^-_t(=\sigma_t)$ this has no consequence. Note also that the particle $\delta_x$ from $\mu_t$ used in that coagulation event will not be removed from $\mu_t$. Similar interpretations of the terms $K(\mu_t,\sigma^-_t)$ and $\frac{1}{2}K'_\pm(\mu_t,\mu_t)$ lead us to define the following Markov chain $\Theta_t = X_t\oplus Y_t\oplus Z_t$ on $\mcN$. Denote by $\Theta_0 = \displaystyle{\Bigl(\sum_{i=1..m}\delta_{x_i}\Bigr)\oplus\Bigl(\sum_{k=1..p}\delta_{y_k}\Bigr)\oplus\Bigl(\sum_{\ell=1..q}\delta_{z_\ell}\Bigr)}$ its starting point.

\subsubsection{Dynamics}
\label{Dynamics}
Associate to each pair

$\bullet$ $1\leq i<j\leq m$, exponential random variables $R_{ij}, S_{ij}$ and $T_{ij}$ with respective parameters $K(x_i,x_j)$ and $K'_+(x_i,x_j)$ and $K'_-(x_i,x_j)$,

$\bullet$ $(i,k)\in \llbracket 1,m\rrbracket\times \llbracket 1,p\rrbracket$ an exponential random variable $U_{ik}$ with parameter $K(x_i,y_k)$,

$\bullet$ $(i,\ell)\in \llbracket 1,m\rrbracket\times \llbracket 1,q\rrbracket$ an exponential random variable $V_{i\ell}$ with parameter $K(x_i,z_\ell)$.

\noindent All these random variables are supposed to be independent. Denoting by $W$ the first event happening in the system
$$
W = \min \bigl\{R_{ij}, S_{ij}, T_{ij}, U_{ik}, V_{i\ell}\,\,; \,\,1\leq i< j\leq m,\,\, k\in\llbracket 1,p\rrbracket,\,\,\ell\in\llbracket 1,q\rrbracket\bigr\},
$$
the jump $\Delta\Theta$ of the Markov chain depends on which of these exponential clocks rings first. For future reference, the different types of events that can happen are numbered. If
\begin{displaymath}
\begin{array}{llll}
W=R_{ij},    & \textrm{then} & \Delta\Theta = \bigl(\delta_{x_i+x_j}-\delta_{x_i}-\delta_{x_j}\bigr)\oplus 0\oplus 0 & \textrm{(event type: } 0 \,\,\,) \\
W=S_{ij},    & \textrm{then} & \Delta\Theta = 0 \oplus\delta_{x_i+x_j}\oplus\bigl(\delta_{x_i}+\delta_{x_j}\bigr) & \textrm{(event type: } 1^+) \\
W=T_{ij},    & \textrm{then} & \Delta\Theta = 0\oplus \bigl(\delta_{x_i}+\delta_{x_j}\bigr)\oplus\delta_{x_i+x_j} & \textrm{(event type: } 1^-) \\
W=U_{ik},    & \textrm{then} & \Delta\Theta = 0\oplus\bigl(\delta_{x_i+y_k}-\delta_{y_k}\bigr)\oplus\delta_{x_i} & \textrm{(event type: } 2^+) \\
W=V_{i\ell}, & \textrm{then} & \Delta\Theta = 0\oplus \delta_{x_i}\oplus\bigl(\delta_{x_i+z_\ell}-\delta_{z_\ell}\bigr) & \textrm{(event type: } 2^-) \\
\end{array}
\end{displaymath}
The process $\Theta_t$ will be constant on the time interval $[0,W)$ and have jump $\Delta\Theta$ at time $W$. The dynamics then starts afresh.

\smallskip

\noindent \textbf{Remark.}
\label{Remark_label}
\emph{It is clear from this description that for any function $\varphi$ satisfying the relation $\varphi(a+b)\geq\varphi(a)-\varphi(b)$ for any $a,b >0$, the function $(\varphi,Y_t+Z_t)$ increases with time. This fact is useful for the convergence result stated in theorem \ref{ThmConvergence}}.

Given any positive integer $N$, define $\frac{1}{N}\Theta_t$ as the element $\frac{1}{N} X_t\oplus \frac{1}{N} Y_t\oplus \frac{1}{N} Z_t$ of $\mcN$, and set
$$
\Theta^N_t := \frac{1}{N}\Theta_{\frac{t}{N}} =: \mu_t^N\oplus\sigma_t^{+,N}\oplus\sigma_t^{-,N}.
$$
Note that the first component of $\Theta^N_t$ is the usual Marcus-Lushnikov process. Set $\sigma_t^N = \sigma_t^{+,N}-\sigma_t^{-,N}$. We are going to prove in theorem \ref{ThmConvergence} that $\sigma_t^N$ converges in law to the sensitivity $\sigma_t$. \textit{Those who do not care about the mathematical details of such a statement can skip  the remaining of this section and go to section \ref{SectionAlgo}}.

\subsubsection{Generator}
\label{Generator}

The analytic description of the Markov chain $\{\Theta^N_t\}_{t\geq 0}$ in terms of its generator will be useful in proving theorem \ref{ThmConvergence}. Given a non-negative measure $\mu$ of the form $\frac{1}{N}\sum\delta_{x_i}$ define the rescaled counting measure on ordered pairs of masses of distinct particles by
$$
\widetilde\mu(A\times A') := \mu(A)\,\mu(A') - \frac{1}{N}\mu(A\cap A'),
$$
and define the measure ${\bf G}^{(N)}(\mu)$ and the operator ${\bf P}^{(N)}(\mu)$ setting for any measurable bounded function $f$
\begin{equation*}
\begin{split}
&\Bigl(f,{\bf G}^{(N)}(\mu)\Bigr) = \frac{1}{2}\int\bigl\{f(x+x')-f(x)-f(x')\bigr\}\,K(x,x')\,\widetilde\mu(\textrm{d}x,\textrm{d}x')
\\
&\Bigl(f,{\bf P}^{(N)}(\mu)\Bigr) = \frac{1}{2}\int\bigl\{f(x+x')-f(x)-f(x')\bigr\}^2 \, K(x,x')\, \widetilde\mu(\textrm{d}x,\textrm{d}x').
\end{split}
\end{equation*}
Given $x>0$ and a non-negative measure $\gamma$ on $\RR^*_+$ we shall write $K(x,\gamma)$ for the integral $\int K(x,y)\gamma(\textrm{d}y)$.

\smallskip

Denote by ${\bf H}^{(N)}$ the generator of the process $\bigl\{\Theta_t^N\bigr\}_{0\leq t\leq T}$; for any bounded measurable functions $f,g,h$ on $(0,\infty)$ the $\RR^3$-valued process
\begin{equation*}
\begin{split}
M^{f,g,h\,;\,N}_t := \bigl(f\oplus g\oplus h, \Theta_t^N\bigr) - \bigl(f\oplus g\oplus h,\Theta_0^N\bigr) - \int_0^t \Bigl(f\oplus g\oplus h,{\bf H}^{(N)}\bigl(\Theta^N_s\bigr)\Bigr)\,\textrm{d}s
\end{split}
\end{equation*}
is a martingale (with respect to its natural filtration). For a measure $\mu$ of the form $\frac{1}{N}\sum \delta_{x_i}$ and $\Theta = \mu\oplus\sigma^+\oplus\sigma^-\in\mcN$ we have
\begin{align}
\label{GeneratorTheta}
\lefteqn{
\Bigl(f\oplus g\oplus h,{\bf H}^{(N)}(\Theta)\Bigr) = }\nonumber\\
&&
\begin{split}
& \hspace{0.2cm}\bigl(f,{\bf G}^{(N)}(\mu)\bigr) \quad \oplus \nonumber\\
& \hspace{0.2cm}\left\{\frac{1}{2}\int \Bigl\{g(x+x')K'_+(x,x')+\bigl(g(x)+g(x')\bigr)K'_-(x,x')\Bigr\}\widetilde\mu(\textrm{d}x,\textrm{d}x')
\right.\\
& \hspace{0.2cm}+ \left.\int\Bigl\{\bigl(g(x+y)-g(y)\bigr)K(x,y)\sigma^+(\textrm{d}y) + g(x)K(x,\sigma^-)\Bigr\}\mu(\textrm{d}x)\right\}\; \oplus \\
& \hspace{0.2cm}\left\{\frac{1}{2}\int \Bigl\{h(x+x')K'_-(x,x')+\bigl(h(x)+h(x')\bigr)K'_+(x,x')\Bigr\}\widetilde\mu(\textrm{d}x,\textrm{d}x')
\right.\\
& \hspace{0.2cm}+ \left.\int\Bigl\{\bigl(h(x+z)-h(z)\bigr)K(x,z)\sigma^-(\textrm{d}z) + h(x)K(x,\sigma^+)\Bigr\}\mu(\textrm{d}x)\right\}
\end{split}\\
\end{align}

Compare this formula with the description of the dynamics given in the section \ref{Dynamics}.
\begin{romannum}
   \item Event $\{W=R_{ij}\}$ corresponds to the term $\bigl(f,{\bf G}^{(N)}(\mu)\bigr)\oplus 0\oplus 0$;
   \item Event $\{W=S_{ij}\}$ corresponds to the term $\frac{1}{2}\int 0\,\oplus\, g(x+y)\,\oplus\,\bigl(h(x)+h(y)\bigr)\,K'_+(x,y)\widetilde\mu(\textrm{d}x,\textrm{d}y)$; a similar term corresponds to the event $\{W=T_{ij}\}$;
   \item Event $\{W=U_{ik}\}$ corresponds to the term $\int \bigl\{0\,\oplus\,\bigl(g(x+z)-g(z)\bigr)\,\oplus\,h(x)K(x,z)\sigma^+(\textrm{d}z)\bigr\}\mu(\textrm{d}x)$; a similar term corresponds to the event $\{W=V_{i\ell}\}$.
\end{romannum}
The sum of all these terms gives $\left(f\oplus g\oplus h,{\bf H}^{(N)}(\Theta)\right)$.

Following a classical approach, the study of martingales of the form $M^{f,g,h\,;\,N}_\cdot$ will be our main tool in the proof of the convergence theorem. The explicit expression of the bracket of $M^{f,g,h\,;\,N}$ will be useful in that task. We have
$$
\bigl<M^{f,g,h\,;\,N}\bigr>_t = \frac{1}{N}\int_0^t \Bigl(f\oplus g\oplus h, {\bf Q}^{(N)}\bigl(\Theta_s^N\bigr)\Bigr)\textrm{d}s,
$$
where ${\bf Q}^N\bigl(\Theta\bigr)$ is characterised on measures $\Theta$ of the form $\Bigl(\frac{1}{N}\sum\delta_{x_i}\Bigr)\oplus\sigma^+\oplus\sigma^-$ by the formula
\begin{align*}
\lefteqn{
\Bigl(f\oplus g\oplus h,{\bf Q}^{(N)}(\Theta)\Bigr) =}\nonumber\\
&&
\begin{split}
& \hspace{0.2cm}\bigl(f,{\bf P}^{(N)}(\mu)\bigr) \quad \oplus\\
& \hspace{0.2cm}\left\{\frac{1}{2}\int \Bigl\{g(x+x')^2K'_+(x,x')+\bigl(g(x)+g(x')\bigr)^2 \, K'_-(x,x')\Bigr\} \widetilde{\mu}(\textrm{d}x,\textrm{d}x') \right. \nonumber\\
& \hspace{0.2cm}+\left.\int\Bigl\{\bigl(g(x+y)-g(y)\bigr)^2 \,K(x,y) \,\sigma^+(\textrm{d}y) + g(x)^2 \, K(x,\sigma^-)\Bigr\}\,\mu(\textrm{d}x)\right\}\;\oplus \nonumber \\
 & \hspace{0.2cm}\left\{\frac{1}{2}\int \Bigl\{h(x+x')^2K'_-(x,x')+\bigl(h(x)+h(x')\bigr)^2 K'_+(x,x')\Bigr\}\widetilde\mu(\textrm{d}x,\textrm{d}x') \right. \nonumber\\
 & \hspace{0.2cm}+\left.\int\Bigl\{\bigl(h(x+z)-h(z)\bigr)^2 \, K(x,z)\sigma^-(\textrm{d}z) + h(x)^2 \, K(x,\sigma^+)\Bigr\}\mu(\textrm{d}x)\right\}
\end{split}
\end{align*}

\subsection{Convergence theorem}
\label{SectionCvgceThm}

Denote by $\mcU$ a bounded open set of some $\RR^d$ indexing the family $K_\la$ of kernels. Let $\varphi : (0,\infty)\rightarrow \RR_+$ be a sublinear function: $\varphi(s x) \leq s\varphi(x)$ for any $s>0$ and $x\in (0,\infty)$; such a function is also subadditive: $\varphi(x+y)\leq\varphi(x)+\varphi(y)$, for any $x,y\in (0,\infty)$. We shall suppose that the interaction kernels $K_\la$ satisfy the growth condition
$$
K_\la(x,y) \leq \varphi(x)\varphi(y)
$$
for any $x,y \in (0,\infty),\,\la\in\mcU$, and that the initial condition of Smoluchowski equation \eqref{SmoDiscreteSymm} (or better its `continuous mass version') satisfies the moment condition
\begin{equation}
\label{ConditionMoment}
\int \varphi(x)^{4+\epsilon} \mu_0(\textrm{d}x)<\infty
\end{equation}
for some (small) $\epsilon>0$. We shall suppose in theorem \ref{ThmConvergence} that $\varphi^2$ is sub-additive; together with the above moment condition \eqref{ConditionMoment} on $\mu_0$ this implies that Smoluchowski equation has a unique strong solution\footnote{In the sense defined in \cite{Nor99}.}, defined for all non-negative times.

The following norm was used on the space $\mcM_1$ of signed Borel measures $\mu$ \st $\|\mu\|_1:=\bigl(\varphi,|\mu|\bigr)<\infty$, in the article \cite{BailleulSensitivity} where the following key result about sensitivity is proved.

\begin{theorem}
\label{ThmSensitivity2}
Assume the moment condition \eqref{ConditionMoment} and that $K_\la(x,y)$ and $\big|K'_\la(x,y)\big|$ are both bounded above by $\varphi(x)\varphi(y)$ for any $x,y$. Then the map $(t,\la)\in [0,\infty)\times\mathcal{U}\mapsto \mu_t^{\la}\in\bigl(\mathcal{M}_1,\,\|.\|_1\bigr)$, is a $\mathcal{C}^1$ function and its derivative $\sigma_t^\la$ satisfies the following equation for any bounded measurable function $f$(\footnote{We write here $\{f\}(x,y)$ for $f(x+y)-f(x)-f(y)$.}).
\begin{align*}
\left(f,\sigma_t^{\la}\right)
&= \left(f,\sigma_0^{\la}\right) + \int_0^t \int \{f\}(x,y)K_\la(x,y)\mu^{\la}_s(dx)\sigma_s^{\la}(dy)ds\\
&+ \frac{1}{2}\int_0^t\int \{f\}(x,y)K'_\la(x,y)\mu^{\la}_s(dx)\mu_s^{\la}(dy)ds
\end{align*}
The function $\sigma_\cdot^\la$ is the only $\bigl(\mathcal{M}_1,\,\|.\|_1\bigr)$-valued solution of this equation.
\end{theorem}

\smallskip

We shall consider here a weaker topology than the $\|\cdot\|_1$-topology. We shall equip the space $\RR_+^{\oplus 3}$ with the $\ell^1$-distance: $\|x\oplus y\oplus z - x'\oplus y'\oplus z'\| := |x-x'|+ |y-y'| + |z-z'|$. Write $\mcM^{\oplus 3}$ for the set of non-negative $\RR_+^{\oplus 3}$-valued measures on $\RR^*_+$, and let $d$ be any distance on $\mcM^{\oplus 3}$ metrising weak convergence: $\{\Theta_n\}_{n\geq 0}$ converges to $\Theta_\infty$ iff for any bounded continuous functions $f,g,h$ on $\RR^*_+$, we have $\bigl(f\oplus g\oplus h, \Theta_n\bigr) \rightarrow\bigl(f\oplus g\oplus h, \Theta_\infty\bigr)$. The space $\bigl(\mcM^{\oplus 3},d\bigr)$ is a Polish space with $\mcN$ as a dense subset.

Fix a positive time $T$. We shall state our convergence theorem in the functional setting $\mcD\bigl([0,T],(\mcM^{\oplus 3},d)\bigr)$ of c\`adl\`ag paths from $[0,T]$ to $(\mcM^{\oplus 3},d)$. This space will be equipped with its Skorokhod topology, for which we refer the reader to the books \cite{Billingsley} or \cite{Pollard} of Billingsley and Pollard. Last, we shall denote by $d_0$ any distance on the set of all non-negative Borel measures on $(0,\infty)$ metrising the following notion of convergence\footnote{This notion of convergence, usually called vague convergence, is weaker than weak convergence.}: $\{\mu_n\}_{n\geq 0}$ converges to $\mu_\infty$ iff we have $(f,\mu_n)\rightarrow (f,\mu_\infty)$ for any bounded continuous measurable function $f$ \textit{with bounded support}.

\smallskip

The starting point $\Theta_0^N$ of $\Theta^N_\cdot$ will be of the form $\frac{1}{N}X_0^N\oplus\frac{1}{N}Y_0^N\oplus\frac{1}{N}Z_0^n$ for some non-negative integer-valued finite measures $X_0^N, Y_0^N, Z_0^N$ on $(0,\infty)$. To shorten the notation we shall denote by
$$
\Theta^N_t =: \mu_t^N\oplus\sigma^{+,N}_t\oplus\sigma^{-,N}_t
$$
the process starting from $\Theta_0^N$ constructed in section \ref{SectionChainGenerator} and corresponding to a given parameter $\la$.

\smallskip

We shall suppose that the function $\varphi$ controlling the kernels $K_\la$ satisfies identity \eqref{CondVarphi1} below. As noted in the remark on page \pageref{Remark_label}, this hypothesis implies that the function $\bigl(\varphi,\sigma^{+,N}_t+\sigma^{-,N}_t\bigr)$ increases with time; this fact will enable us to control $\Theta^N$. Note that this hypothesis is weaker than requiring that $\varphi$ be increasing.

\smallskip

\begin{theorem}[Convergence of the particle system]
\label{ThmConvergence}
Let $K_\la(\cdot,\cdot) : \RR^*_+\times\RR^*_+ \rightarrow [0,+\infty)$ be a family of symmetric kernels indexed by $\la\in\mcU$. We suppose the map $(\la\,;\,x,x')\mapsto K_\la(x,x')$ continuous and differentiable \wrt $\la$, with a derivative $K_\la'(x,x')$ continuous \wrt $(x,x')$. Let $\varphi\geq 1$ be a subadditive function whose square is also subadditive. Assume that
\begin{equation}
\label{CondVarphi1}
\varphi(a+b)\geq\varphi(a)-\varphi(b), \quad \textrm{for any positive }a,b,
\end{equation}
\begin{equation}
\label{CondVarphi2}
\begin{split}
\forall\,\la\in\mcU, \forall\,x,x',y\in\RR^*_+, \quad &K_\la(x,x') \leq \varphi(x)\,\varphi(x'), \\
                                                                 &\bigl|K_\la'(x,y)\bigr| \leq \varphi(x)\,\varphi(y), \\
\end{split}
\end{equation}
\begin{equation}
\label{CondVarphi3}
\frac{K_\la(x,x')}{\varphi(x)\,\varphi(x')}\;\;\textrm{ and } \;\;\frac{K'_\la(x,x')}{\varphi(x)\,\varphi(x')} \underset{x+x'\rightarrow \infty}{\longrightarrow} 0
\end{equation}
Fix $\la\in\mcU$ and write $\Theta^N_\cdot$ for the corresponding process in $\mcN$, started from $\mu_0^N\oplus\sigma^{+,N}_0\oplus\sigma^{-,N}_0$. Suppose that $\mu_0$ satisfies the moment condition \eqref{ConditionMoment} for some (small) $\epsilon$, that
\begin{equation}
\label{CondConvMu0}
d_0\Bigl(\varphi\mu_0^N,\varphi\mu_0\Bigr) \rightarrow 0,
\end{equation}
and that there exists a positive constant $C$ bigger than $\bigl(\varphi^2,\mu_0^N\bigr) $ and $\bigl(\varphi,\sigma^{+,N}_0+\sigma^{-,N}_0\bigr)$ for any $N\geq 1$.

Then the sequence of the laws of the processes $\Theta^N$ is tight and any (random) weak limit is almost surely of the form $\bigl\{\mu_t\oplus \sigma^{+,\infty}_t\oplus \sigma^{-,\infty}_t\bigr\}_{0\leq t\leq T}$, with
$$
\sigma^{+,\infty}_t-\sigma^{-,\infty}_t = \sigma_t.
$$
\end{theorem}

\begin{proof}
The following estimate is essential in controlling the behaviour of the processes $\sigma^{+,N}$ and $\sigma^{-,N}$.

\smallskip

\begin{lemma}
\label{LemmeControlSigma}
There exists a positive constant $C_1$ \st
$$
\EE\Bigl[\,\underset{0\leq t\leq T}{\sup} \;\bigl(\varphi,\sigma^{+,N}_t+\sigma^{-,N}_t\bigr)\Bigr] \leq C_1.
$$
\end{lemma}

First decompose $\bigl(\varphi,\sigma^{+,N}_t+\sigma^{-,N}_t\bigr)$ as the sum of a martingale $\bigl\{M_t\bigr\}_{0\leq t\leq T}$ and a finite variation term:


\begin{align*}
\lefteqn{\bigl(\varphi,\sigma^{+,N}_t+\sigma^{-,N}_t\bigr) = \bigl(\varphi,\sigma^{+,N}_0+\sigma^{-,N}_0\bigr) + M_t} \nonumber\\
& \quad\quad\quad + \int_0^t \left(\int \bigl\{\varphi(x+x')+\varphi(x)+\varphi(x')\bigr\}\,K'(x,x')\,\widetilde{\mu}^N_s(\textrm{d}x,\textrm{d}x')\right. \nonumber\\
& \quad\quad\quad \left. +\int \bigl\{\varphi(x+y)-\varphi(y)+\varphi(x)\bigr\}\,K(x,y)\,\mu^N_s(\textrm{d}x)\bigl(\sigma^{+,N}_s+\sigma^{-,N}_s\bigr)(\textrm{d}y)\right)\,\textrm{d}s.
\end{align*}

From \eqref{CondVarphi2} we have for each $N\geq 1$ and $t\in [0,T]$
\begin{equation*}
\begin{split}
\bigl(\varphi,\sigma^{+,N}_t+\sigma^{-,N}_t\bigr) \leq C + M_t \,+ &\int_0^t \int 2\bigl\{\varphi(x)+\varphi(x')\bigr\}\varphi(x)\varphi(x')\mu^N_s(\textrm{d}x)\mu_s^N(\textrm{d}x')\,\textrm{d}s \\
+ &\,2 \int_0^t\bigl(\varphi^2,\mu^N_s\bigr)\,\bigl(\varphi,\sigma^{+,N}_s+\sigma^{-,N}_s\bigr)\,\textrm{d}s.
\end{split}
\end{equation*}
This upper bound is simplified using the subadditivity of $\varphi$ and $\varphi^2$ from which we have\footnote{Since $\varphi\geq 1$ we have $\bigl(\varphi,\mu_0^N\bigr)\leq\bigl(\varphi^2,\mu_0^N\bigr)\leq C$.}
$$
(\varphi,\mu^N_t) \leq (\varphi,\mu^N_0) \leq C \quad\textrm{and}\quad (\varphi^2,\mu^N_t) \leq (\varphi^2,\mu^N_0) \leq C.
$$
This gives a Gr\"onwall-type inequality
$$
\bigl(\varphi,\sigma^{+,N}_t+\sigma^{-,N}_t\bigr) \leq C + M_t + 4C^2T + 2C\int_0^t\bigl(\varphi,\sigma^{+,N}_s+\sigma^{-,N}_s\bigr)\,\textrm{d}s
$$
whose mean version gives a constant $C_1$ \st $\EE\Bigl[\bigl(\varphi,\sigma^{+,N}_t+\sigma^{-,N}_t\bigr)\Bigr] \leq C_1$ for any $0\leq t \leq T$. We get the statement of the lemma recalling that hypothesis \eqref{CondVarphi1} implies that the function $t\mapsto \bigl(\varphi,\sigma^{+,N}_t+\sigma^{-,N}_t\bigr)$ is increasing.

\smallskip

Given $\epsilon>0$ define the compact subset
$$
K_\epsilon = \Bigl\{\mu\oplus\sigma^+\oplus\sigma^-\in\mcM^{\oplus 3}\,;\,\max \bigl\{(\varphi,\mu),\,\bigl(\varphi,\sigma^+\bigr),\,\bigl(\varphi,\sigma^-\bigr)\bigr\}\leq \frac{1}{\epsilon}\Bigr\}\subset \mcM^{\oplus 3},
$$
and denote by $\PP^N$ the law of $\Theta^N_\cdot$ on $\mcD\Bigl([0,T],\bigl(\mcM^{\oplus 3},d\bigr)\Bigr)$.

\begin{corollary}[Compactness]
Given $\eta>0$, there exists $\epsilon>0$ \st
$$
\PP^N\Bigl(D\bigl([0,T],K_\epsilon\bigr)\Bigr)\geq 1-\eta.
$$
\end{corollary}

\medskip

Now let $f,g,h$ be bounded measurable functions on $(0,\infty)$ no greater than $1$. By lemma \ref{LemmeControlSigma} we have for all $s<t$
\begin{align*}
\lefteqn{
\EE\left[\int_s^t \bigl\|\bigl(f\oplus g\oplus h, {\bf H}^{(N)}\bigl(\Theta_s^N\bigr)\bigr)\bigr\|\textrm{d}s\right]}\\
& \hspace{0.9cm} \leq 2C^2(t-s) + 2\int_s^t\EE\left[\frac{3C^2}{2}+2C\,\bigl(\varphi,\sigma_r^{+,N}+\sigma_r^{-,N}\bigr)\right]\,\textrm{d}r \\
& \hspace{0.9cm} \leq C_2(t-s)
\end{align*}
and
\begin{align*}
\lefteqn{
\EE\left[\bigl<M^{f,g,h\,;\,N}\bigr>_t - \bigl<M^{f,g,h\,;\,N}\bigr>_s\right]}\\
& \hspace{0.9cm}\leq \frac{1}{N}\EE\left[\int_s^t \bigl\|\bigl(f\oplus g\oplus h, {\bf Q}^{(N)}\bigl(\Theta_r^N\bigr)\bigr)\bigr\|\textrm{d}s\right] \\
& \hspace{0.9cm}\leq \frac{4C^2}{N}+\frac{1}{N}\int_s^t 2 \EE\Bigl[\frac{C^2 + 4C^2}{2}+ 4C\,\bigl(\varphi,\sigma^{+,N}_r\bigr)+C\bigl(\varphi,\sigma^{-,N}_r\bigr)\Bigr]\,\textrm{d}r \\
& \hspace{0.9cm}\leq \frac{C_2}{N}(t-s),
\end{align*}

where $C_2$ is a positive constant depending only on $C$. So, by Doob's $\LL^2$-inequality, we have
\begin{equation}
\label{EstimateVariation}
\EE\left[\underset{s\leq r\leq t}{\sup} \bigl\|\bigl(f\oplus g\oplus h,\Theta_r^N-\Theta_s^N\bigr)\bigr\|^2\right] \leq C_3\Bigl((t-s)^2+\frac{t-s}{N}\Bigr)
\end{equation}
for some positive constant $C_3$ depending only on $C$. It is then a standard fact that the equicontinuity inequality \eqref{EstimateVariation} together with corollary on compactness enable the use of Jakubowski's criterion\footnote{See for instance Dawson's lecture notes \cite{Dawson}.}; so the sequence of laws of $\Theta^N_\cdot$  in $\mcD\Bigl([0,T],\bigl(\mcM^{\oplus 3},d\bigr)\Bigr)$ has a convergent subsequence. Denote by $\Theta^\infty_\cdot = \mu^\infty\oplus\sigma^{+,\infty}\oplus\sigma^{-,\infty}$ any limit point. Taking a subsequence and changing the probability space if necessary we can suppose \wlg that $\Theta^N_\cdot$ converges almost surely to $\Theta^\infty_\cdot$ in $\mcD\Bigl([0,T],\bigl(\mcM^{\oplus 3},d\bigr)\Bigr)$. As $\Theta^N_\cdot$ makes jumps of size at most $\frac{3}{N}$, in the total variation distance, the limit process is a continuous process from $[0,T]$ to $\bigl(\mcM^{\oplus 3},d\bigr)$.

\smallskip

It is proved in \cite{Nor99} that under conditions \eqref{CondConvMu0} and \eqref{ConditionMoment} the process $\mu^\infty_\cdot$ is almost surely equal to the unique strong solution $\mu_\cdot$ of Smoluchowski equation, and that we have almost surely $\underset{s\leq t}{\sup}\;d_0\bigl(\varphi\mu^N_s,\varphi\mu_s\bigr)\rightarrow 0$, as $N$ goes to $\infty$.

\medskip

To prove that $\sigma_\cdot^{+,\infty} - \sigma_\cdot^{-,\infty}$ is equal to the unique solution of equation \eqref{SensitivityEq1} it suffices to prove that it satisfies this equation for any bounded measurable function $g$ with compact support, as a straightforward limit argument will give it for any bounded measurable function. We shall suppose \wlg that $\sigma^{+,N}_0-\sigma^{-,N}_0 = 0$. We shall adopt the notation
$$
\sigma^N_s := \sigma^{+,N}_s-\sigma^{-,N}_s, \quad \big|\sigma^N_s\big| := \sigma^{+,N}_s+\sigma^{-,N}_s
$$
and
$$
\sigma^\infty_s := \sigma^{+,\infty}_s-\sigma^{-,\infty}_s, \quad \big|\sigma^\infty_s\big| := \sigma^{+,\infty}_s+\sigma^{-,\infty}_s.
$$
The conclusion of lemma \ref{LemmeControlSigma} can now be re-written as
$
\EE\Bigl[\,\underset{0\leq t\leq T}{\sup} \;\bigl(\varphi,\big|\sigma^N_t\big|\bigr)\Bigr] \leq C_1$.

\bigskip

It can be seen from expression \eqref{GeneratorTheta} for ${\bf H}^{(N)}$ that the real-valued process
\begin{equation}
\label{FormulaB}
\begin{split}
B_t^{g\,;\,N} = \bigl(g,\sigma^N_t\bigr) &- \int_0^t\left( \int\frac{1}{2}\bigl\{g(x+x')-g(x')-g(x)\bigr\}\,K'(x,x')\,\widetilde{\mu}_s^N(\textrm{d}x,\textrm{d}x')\right. \\
& + \left.\int \bigl\{g(x+y)-g(y)-g(x)\bigr\}\,K(x,y)\,\mu_s^N(\textrm{d}x)\,\sigma^N_s(\textrm{d}y)\right)\textrm{d}s
\end{split}
\end{equation}
is a martingale with previsible increasing process
\begin{equation*}
\begin{split}
\bigl<B^{g\,;\,N}\bigr>_t = &\,\frac{1}{N}\int_0^t \left( \int\frac{1}{2}\bigl\{g(x+x')-g(x')-g(x)\bigr\}^2 \, K'(x,x')\, \widetilde{\mu}_s^N(\textrm{d}x,\textrm{d}x')\right. \\
+ &\left. \int \bigl\{g(x+y)-g(y)-g(x)\bigr\}^2 \, K(x,y)\, \mu_s^N(\textrm{d}x)\, \sigma^N_s(\textrm{d}y)\right)\textrm{d}s
\end{split}
\end{equation*}
Using lemma \ref{LemmeControlSigma} together with the almost sure inequality $\bigl(\varphi,\mu^N_s\bigr)\leq C$, it is seen that $\EE\bigl<B^{g\,;\,N}\bigr>_T$ converges to $0$ as $N$ goes to $\infty$. So, to show that $\sigma^\infty_\cdot$ satisfies equation \eqref{SensitivityEq1}, it is sufficient to prove that the two integrals inside the right hand side of equation \eqref{FormulaB} converge \as to
$$
\int \frac{1}{2}\bigl\{g(x+x')-g(x')-g(x)\bigr\} \, K'(x,x')\,\mu_s(\textrm{d}x)\,\mu_s(\textrm{d}x')
$$
and
\begin{equation}
\label{SecondTerm}
\int \bigl\{g(x+y)-g(y)-g(x)\bigr\} \, K(x,y)\,\mu_s(\textrm{d}x)\,\sigma^\infty_s(\textrm{d}y)
\end{equation}
respectively, and that we have uniform bounds on them so that dominated convergence under the time integral can be used. The convergence of the first integral was proved in \cite{Nor99} using hypotheses \eqref{CondVarphi2} and \eqref{CondVarphi3}, with $K$ in place of $K'$; the same argument applies here. This integral is bounded above by $\frac{3}{2}\|g\|_\infty C^2$, uniformly in $s\in [0,T]$ and $N\geq 1$.

\smallskip

Given $\delta\in (0,\infty]$, the function $\varphi^\delta(x) = \varphi(x){\bf 1}_{x\leq\delta}$ is subadditive. It comes from Fatou's lemma that the inequality
$$
\EE\Bigl[\,\underset{0\leq t\leq T}{\sup} \;\bigl(\varphi^\delta,\big|\sigma^\infty_T\big|\bigr)\Bigr] \leq C_1
$$
holds for any $\delta\in(0,\infty]$. So, to any $\omega\in\Omega$ one can associate a positive constant $m(\delta\,;\,\omega)$ \st we have
$$
\bigl(\varphi^\delta,\big|\sigma^\infty_t(\omega)\big|\bigr) \leq \bigl(\varphi^\delta,\big|\sigma^\infty_T(\omega)\big|\bigr) \leq m(\delta\,;\,\omega)
$$
on the time interval $[0,T]$. One can choose this constant $m(\delta\,;\,\omega)$ so that it converges to $0$ as $\delta$ decreases to $0$. Taking $\omega$ in a subset $\Omega_1$ of $\Omega$ of probability $1$, for which $\Theta^N_\cdot(\omega)$ converges to $\Theta^\infty_\cdot(\omega)$ in $\mcD\bigl([0,T],\bigl(\mcM^{\oplus 3},d\bigr)\bigr)$, we get that
$$
\bigl(\varphi^\delta,\big|\sigma^N_t(\omega)\big|\bigr) \leq \bigl(\varphi^\delta,\big|\sigma^N_T(\omega)\big|\bigr)
$$
is arbitrarily small provided $\delta$ is small enough, and bounded above uniformly in $t\in [0,T], N\geq 1$ and $\delta\in (0,\infty]$.

Proceed now as in \cite{Nor99} and write $K$ as the sum of a kernel $K_1$ with compact support and a kernel $K_2$ with support in
$$
F_1\cup F_2\cup F_3 := \bigl\{(x,y)\,;\,x\leq\delta\bigr\}\cup \bigl\{(x,y)\,;\,y\leq\delta\bigr\}\cup \Bigl\{(x,y)\,;\,\max \{x,y\}\geq \frac{1}{\delta}\Bigr\}.
$$
There is no problem in justifying the convergence of the integral in \eqref{SecondTerm} corresponding to $K_1$. For $K_2$ write, with $\{g\}(x,y) := g(x+y)-g(x)-g(y)$,
\begin{align*}
\lefteqn{
\left| \int\{g\}(x,y)\,K_2(x,y)\,\Bigl(\mu_s^N(\textrm{d}x)\,\sigma^N_s(\textrm{d}y) - \mu_s(\textrm{d}x)\,\sigma^\infty_s(\textrm{d}y)\Bigr)\right|}\\
& \hspace{0.9cm}\leq \left|\int\{g\}(x,y)\,K_2(x,y)\,\bigl(\mu_s^N-\mu^\infty_s\bigr)(\textrm{d}x)\,\sigma^N_s(\textrm{d}y)\right|\\ & \hspace{0.9cm} + \left|\int\{g\}(x,y)\,K_2(x,y)\,\mu_s(\textrm{d}x)\,\bigl(\sigma^N_s-\sigma^\infty_s\bigr)(\textrm{d}y)\Bigr)\right|
\end{align*}
and deal with each term of the right hand side separately. The first term is bounded above by $d_0\bigl(\varphi\mu_s^N(\omega),\varphi\mu_s\bigr)\bigl(\varphi,\big|\sigma_s^N(\omega)\big|\bigr)$, up to a multiplicative constant. As the first factor converges to $0$ (and is no greater than $2C$) while the second is uniformly bounded above, one can apply dominated convergence in the corresponding integral \wrt $s$. To deal with the second term, use the pointwise bounds\footnote{$\|\cdot\|_0$ denotes total variation norm.}
\begin{equation*}
\begin{split}
&\bigl\|K_2{\bf 1}_{F_1}\mu_s\oplus\sigma^N_s(\omega)\bigr\|_0 \leq \gamma_\delta\,C\,\bigl(\varphi,\big|\sigma^N_s\big|(\omega)\bigr), \\
&\bigl\|K_2{\bf 1}_{F_2}\mu_s\oplus\sigma^N_s(\omega)\bigr\|_0 \leq C\,\bigl(\varphi^\delta,\big|\sigma^N_s\big|(\omega)\bigr), \\
& \bigl\|K_2{\bf 1}_{F_3}\mu_s\oplus\sigma^N_s(\omega)\bigr\|_0 \leq \bigl(\varphi^\delta,\mu_s\bigr)\bigl(\varphi,\big|\sigma^N_s\big|(\omega)\bigr), \\
\end{split}
\end{equation*}
where $\gamma_\delta = \max\Bigl\{\frac{K(x,y)}{\varphi(x)\varphi(y)}\,;\,(x,y)\in F_3\Bigr\}$ converges to $0$ as $\delta$ decreases to $0$. As $\bigl(\varphi,\big|\sigma^N_s(\omega)\big|\bigr)$ is uniformly bounded above by a constant, and both $\bigl(\varphi^\delta,\big|\sigma^N_s(\omega)\big|\bigr)$ and $\bigl(\varphi^\delta,\mu_s\bigr)$ can be made arbitrarily small for small enough $\delta$, we have enough control to apply dominated convergence.
\qquad
\end{proof}


\section{Algorithm}
\label{SectionAlgo}

We describe in this section the algorithm used to simulate the particle system studied above; the numerical results are to be found in section \ref{SectionNumericalResults}. Two points of computational interest are first put forward in sections \ref{SectionCoupling} and \ref{SectionMajKernel}; the algorithm itself is described in section \ref{Algo}.

\subsection{Coupling}
\label{SectionCoupling}

The basic algorithm to simulate the sensitivity $\sigma_t$ is given by the dynamics of the process $\Theta^N$ described in section \ref{SectionChainGenerator}. A fresh look at it reveals a potential computational drawback of this approach: It is seen from the explicit expression \eqref{GeneratorTheta} of the generator of $\Theta^N$ that the mean number of particles inside $\sigma^N$ satisfies a Gr\"onwall-type inequality, which implies an exponential growth of this quantity. One should see in this exponential growth of the number of particles a good feature for the approximation qualities of our estimator $\sigma_t^N$ of $\sigma_t$, especially regarding accuracy and variance. This should be opposed to what happens for the weighed and coupled particles systems described in the introduction, for which the number of particles in the system decreases with time\footnote{This decrease is of the same order for the weighted particle system and for Marcus-Lushnikov's dynamics; it is worse for the coupled system. In this approach, $\sigma_t$ is approximated by the ratio $(\mu_t^{\la+\frac{1}{2}\delta\la\,;\,N}-\mu_t^{\la-\frac{1}{2}\delta\la\,;\,N})/ \delta\la$, where $\mu_t^{\la+\frac{1}{2}\delta\la\,;\,N}$ and $\mu_t^{\la-\frac{1}{2}\delta\la\,;\,N}$ are two coupled Markus-Lushnikov processes. So, the smaller $\delta\la$ is, the more $\mu_t^{\la+\frac{1}{2}\delta\la}$ and $\mu_t^{\la-\frac{1}{2}\delta\la}$ (and $\mu_t^{\la+\frac{1}{2}\delta\la\,;\,N}$ and $\mu_t^{\la-\frac{1}{2}\delta\la\,;\,N}$ with it) look the same. This means that the `real' number of particles in the difference $\mu_t^{\la+\frac{1}{2}\delta\la\,;\,N}-\mu_t^{\la-\frac{1}{2}\delta\la\,;\,N}$ is a `function' $f_{\delta\la}(N)\leq N$ of $\delta\la$ that decreases as $\delta\la$ goes to $0$, a necessary condition for the ratio to be a good estimate of $\sigma_t$.}.

\smallskip

As an exponential growth of the quantity of information to consider is non-desirable for simulations, three kinds of tricks are used in order to reduce it. 
\begin{romannum}
   \item \textbf{Cancellation.} As we are only interested in the difference $\sigma_t^{+,N} - \sigma_t^{-,N}$ any particle which appears in both particle systems will be removed from both of them.
   \item \textbf{Coupling.} A particle $\delta_x$ of $\mu^N$ coagulates with any particle of $\sigma_t^{+,N}$ at rate $\frac{1}{N}K\bigl(x,\sigma_t^{+,N}\bigr) = \frac{1}{N}\int K(x,y)\sigma_t^{+,N}(\textrm{d}y)$; it also coagulates with any particle of $\sigma_t^{-,N}$ at rate $\frac{1}{N}K\bigl(x,\sigma_t^{-,N}\bigr)$. This particle is thus used in both systems at rate $\frac{1}{N}K\bigl(x,\sigma_t^{+,N}\bigr)\wedge K\bigl(x,\sigma_t^{-,N}\bigr)$, in which case a cancellation removes the particles $\delta_x$ added to $\sigma_t^{-,N}$ and $\sigma_t^{+,N}$. This operation leaves the total number of particles in $\sigma^N$ constant. The rest of the time $\delta_x$ is used in only one of the systems.
   \item \textbf{Re-sampling.} A more drastic control of the number of particles in $\sigma^N$ can be obtained using re-sampling. Let $M$ and $m$ be two integers depending on $N$, with $m\leq M$. Each time $\sigma_t^{+,N}$ or $\sigma_t^{-,N}$ has $M$ particles, replace it by an iid sample of itself of size $m$; this way the total number of particles in $\sigma^N$ remains no greater than $2M$.
\end{romannum}

\subsection{Majorant kernel}
\label{SectionMajKernel}

In order to treat information in a computationally efficient way, we have organized the data using tree structures. The use of a majorant kernel with a simple algebraic structure together with an acceptance/rejection step lead to an efficient updating of the data tree.

The choice of a majorant kernel $\Khat(\cdot,\cdot)$ is made so that $\Khat$ is symmetric, no less than $K$ and has the form
\begin{equation}
\label{factorisation_eq}
\Khat(x_i,x_j) = \sum_\beta\Khat_\beta(x_i,x_j) := \sum_\beta f_\beta(x_i) \, g_\beta(x_j)
\end{equation}
for $\beta$ in a finite set of indices \cite{Wagner3}. This form of kernel leads to simple generation of probabilities of the form
\begin{equation}
\label{PbaSampling}
\frac{\Khat(x_i,x_j)}{\sum_{a\neq b} \Khat(x_a,x_b)} = \sum_\beta \frac{\sum_{a\neq b} f_\beta(x_a) \, g_\beta(x_b)}{\sum_{a\neq b} \sum_{\beta'} f_{\beta'}(x_a) \, g_{\beta'}(x_b)}\; \frac{f_\beta(x_i)}{\sum_a f_\beta(x_a)}\;\frac{g_{\beta}(x_j)}{\sum_{b\,;\, b\neq a} g_{\beta}(x_b)},
\end{equation}
where $a$ and $b$ run in possibly different finite sets of indices. Identity \eqref{PbaSampling} corresponds to choosing first an index $\beta$ according to the probability specified by the first term of the \rhs and then choosing each particle $x_i, x_j$ \textit{separately}. The choice of a pair $(x_i,x_j)$ according to the probability given the \lhs of formula \eqref{PbaSampling} can thus be done in $O(N)$ operations rather than $O(N^2)$. All the required information can be held in binary tree structures (as described in \cite{LPDA}) whilst allowing an even further reduction in the number of operations to choose each particle from $O(N)$ to $O(\log N)$. Updating this information also requires $O(\log N)$ operations. Further, the sums in the first fractions of the \rhs of \eqref{PbaSampling} are automatically contained in the tree structure without further computation.

Note that in the theoretical framework used in section \ref{SectionTheoretic}, the function $\varphi(x)\,\varphi(y)$ can be used as a unique majorant kernel. We have yet chosen to present the above general procedure as we shall consider situations in which the above theory does not apply directly.

\subsection{Algorithm description}
\label{Algo}

Recall $\Theta^N_t$ is of the form $\Bigl(\frac{1}{N}X_t,\frac{1}{N}Y_t,\frac{1}{N}Z_t\Bigr)$ for a Markov process $\Theta_t = \bigl(X_t,Y_t,Z_t\bigr)$ whose components are sums of Dirac masses and whose dynamics was described in section \ref{SectionChainGenerator}. What the algorithm really simulates is the discrete measure-valued process $\Theta_t$; a rescaling gives the time evolution of $\Theta^N_t$. The algorithm is described in Algorithms \ref{AlgoPart1} and \ref{AlgoPart2} below.

\medskip

\noindent Note that there may be up to three different majorant kernels --- for $K$, $K_{+}^{'}$ and $K_{-}^{'}$. Therefore, we slice up the total majorant rates according to the event type $\al \in \{ 0, 1^+, 1^-, 2^+, 2^-\}$ to occur. We then have $\widehat{K}_{\alpha \beta}$ such that $\sum_{\beta} \widehat{K}_{\alpha \beta} = \widehat{K}_{\alpha}$ (from \eq \ref{factorisation_eq}), where $\widehat{K}_{\alpha} \in \{\widehat{K}, \widehat{K}_{+}^{'}, \widehat{K}_{-}^{'}\}$. This gives the corresponding rates $\widehat{\rho}_{\alpha \beta}$ and $\widehat{\rho}_{\alpha}$.


\incmargin{2em}
\restylealgo{algoruled} 
\Setvlineskip{0.4cm}
\dontprintsemicolon
\begin{algorithm}[!hb]
\SetVline

\nl Set $t=0$.
\While{$t < t_{\textrm{end}}$}
{
    \nl Generate a realisation of the holding time $\Delta t$ with exponential law of parameter $\frac{1}{N}\sum_{\alpha}\rhohat_{\alpha}$, and set $t \leftarrow t + \Delta t$.\;

    \nl Choose event type $\alpha \in \{ 0, 1^+, 1^-, 2^+, 2^-\}$ to occur with distribution $\frac{\rhohat_{\alpha}}{\sum_{\alpha}\rhohat_{\alpha}}$. \;

    \nl Choose process $\beta$ with distribution $\frac{\rhohat_{\alpha\beta}}{\rhohat_{\alpha}}$.\;

    \nl Given $\alpha$ and $\beta$, choose a pair of particles using the \textbf{index distribution}
    \begin{equation} \frac{\Khat_{\alpha\beta}(x_i,x_j)}{\rhohat_{\alpha\beta}} = \frac{f_{\al\beta}(x_i)}{\sum_af_{\al\beta}(x_a)} \frac{g_{\al\beta}(x_j)}{\sum_bg_{\al\beta}(x_b)}
    \end{equation}
    where $(x_i,x_j)$ are the masses of particles sampled from the appropriate ensembles ($\mu^N$, $\sigma^{+,N}$ or $\sigma^{-,N}$) depending on $\al$.\;

    \nl Perform the coagulation step which depends on $\alpha$:

    \Switch{the value of $\alpha$ chosen}
    {
        \nl \Case{$\alpha = 0$; this part is the original Marcus-Lushnikov process.}
        {
             The chosen pair of particles is of the form $(x_i,x_j)$.\;

            \nl With probability $\frac{K_{\alpha}}{\Khat_{\alpha}}$ make the jump $\Delta\Theta^N = \bigl(\delta_{x_i+x_j}-\delta_{x_i}- \delta_{x_j}\bigr)\oplus 0\oplus 0$.\;
        }
        \lnl{case_1} \Case{$\alpha=1^+\textrm{\emph{or }} 1^-$}
        {
            The chosen pair of particles is of the form $(x_i,x_j)$.
            \nl Set $p =  \frac{\max\{K_{2^+},K_{2^-}\}}{\Khat_{2^+}+\Khat_{2^-}}$, and generate a realisation of a uniform random variable $\textbf{U} $ in $(0,1)$.\;

            \uIf{$\textbf{U} \leq p$}
            {
                \uIf{$K_{2^+} > K_{2^-}$}
                {
                    \nl make the jump $\Delta\Theta^N = 0\oplus\delta_{x_i+x_j}\oplus \delta_{x_i}+\delta_{x_j}$.\;
                }
                \Else
                {
                    \nl make the jump  $\Delta\Theta^N = 0\oplus\delta_{x_i}+\delta_{x_j}\oplus\delta_{x_i+x_j}$.\;
                }
            }
            \lElse
            {
                \nl Go to Step \ref{cancellation_step}.\;
            }
        }
        \lnl{case_2} \lCase{$\alpha=2^+ \textrm{\emph{or }}2^-$;}
 {
            Go to Algorithm \ref{AlgoPart2}.\;
        }
    }

    \lnl{cancellation_step} For \textbf{each} particle of $\sigma^N$ that has just been involved in a coagulation or newly formed, do a cancellation operation if it can be done.\;
}
\nl STOP.\;
\caption{The ExactCoupling algorithm - Part 1}
\label{AlgoPart1}
\end{algorithm}

\incmargin{2em}
\restylealgo{algoruled}
\Setvlineskip{0.4cm}
\dontprintsemicolon
\begin{algorithm}
\SetVline

\lnl{case_2_2ndpart} \Case{$\alpha=2^+$ or $\alpha=2^-$}
        { The chosen ordered pair of particles contains one particle of $\mu^N$ and one particle of $\sigma^N$, in either order.\;

            \uIf{the pair is of the form $(x_i,\cdot)$ where $x_i$ is the mass of a particle from $\mu^N$}
            {
                \uIf{the second particle belongs to $\sigma^{+,N}$}
                {
                    \nl Choose a particle of $\sigma^{-,N}$ according to the distribution
                    \begin{equation}
                    \frac{g_{\al\beta}(\cdot)}{\sum_{\ell \in \llbracket 1,\ldots,q\rrbracket}g_{\al\beta}(z_\ell)}.
                    \end{equation}
                }
                \Else
                {Choose a particle of $\sigma^{+,N}$ according to the distribution
                    \begin{equation}
                    \frac{g_{\al\beta}(\cdot)}{\sum_{k \in \llbracket1,\ldots,p\rrbracket}g_{\al\beta}(y_k)}.
                    \end{equation}
                }
                \nl Set
                \begin{equation}
                r_+ := \displaystyle{\sum_{k \in \llbracket 1,\ldots,p\rrbracket}g_{\al\beta}(y_k)} \quad, \quad r_- := \displaystyle{\sum_{\ell \in \llbracket 1,\ldots,q\rrbracket}g_{\al\beta}(z_\ell)}
                \end{equation}.\;

            }
            \Else
            {
                \nl Do the symmetrical operation, swapping $g_{\al\beta}$ with $f_{\al\beta}$.\;
            }

            \nl The preceding steps produce a triple $(x_i,y_k,z_\ell)$ of particles from $\mu^N \oplus \sigma^{+,N} \oplus \sigma^{-,N}$. Set
            \begin{equation}
            p_{\min} = \frac{\min\{r_+,r_-\}}{r_+ + r_-} \frac{K}{\Khat} \quad, \quad
            p_{\max} = \frac{\max\{r_+,r_-\}}{r_+ + r_-} \frac{K}{\Khat}.
            \end{equation}\;

            \nl Generate realisation of a uniform random variable $\textbf{U}$ in $(0,1)$.\;

            \uIf{$0 < \textbf{U} \leq p_{\min}$}
            {
                \nl make the jump $\Delta\Theta^N = 0\oplus\bigl(\delta_{x_i+y_k}-\delta_{y_k}\bigr)\oplus\bigl(\delta_{x_i+z_\ell}-\delta_{z_\ell}\bigr)$.\;
            }
            \uElseIf{$p_{\min} < \textbf{U} \leq p_{\max}$}
            {
                \uIf{$r_+ > r_-$}
                {
                    \nl make the jump $\Delta\Theta^N = 0\oplus\bigl(\delta_{x_i+y_k}-\delta_{y_k}\bigr)\oplus\delta_{x_i}$.\;
                }
                \Else
                {
                    \nl make the jump $\Delta\Theta^N = 0\oplus\delta_{x_i}\oplus\bigl(\delta_{x_i+z_\ell}-\delta_{z_\ell}\bigr)$.\;
                }
            }
            \Else
            {
                \nl Go to Step \ref{cancellation_step} of Algorithm \ref{AlgoPart1}.\;
            }
        }
        \nl Go to Step \ref{cancellation_step} of Algorithm \ref{AlgoPart1}.\;
\caption{The ExactCoupling algorithm - Part 2 (Cases $\alpha=2^+,2^-$ only)}
\label{AlgoPart2}
\end{algorithm}

\medskip

\noindent To order to describe numerical results it provides, we shall denote by $L$ the number of simulations with the same initial conditions and by $t_{\textrm{run}}$ the computational time taken to run the algorithm (CPU time in seconds).

\section{Numerical Results}
\label{SectionNumericalResults}

We have chosen to illustrate our approach in situations where the theoretical results of section \ref{SectionTheoretic} do not apply, so as to show its robustness. The main motivation of this article is to produce a stochastic estimate of the sensitivity $\sigma_t$ whose variance is smaller than that given by existing methods. One step in this direction was done in \cite{PeterJamesMarkus}, where $\sigma_t$ was approximated by the ratio $(\mu_t^{\la+\frac{1}{2}\delta\la\,;\,N} - \mu_t^{\la-\frac{1}{2}\delta\la\,;\,N})/\delta\la$, for two Marcus-Lushnikov processes with slightly different parameters. The method there called for coupling them so as to reduce the variance of this estimator as much as can be done; this was done in the same spirit as the coupling used above. We shall refer to this algorithm as the \textbf{CD} algorithm (for central difference). The variance reduction obtained by this method is significant; we shall thus compare our results with those given by the CD algorithm. As our algorithm simulates $\sigma_t$ directly, it will be called \textbf{Exact}; and depending on whether or not we use the coupling step we shall talk of the \textbf{ExactCoupling} or \textbf{ExactIndep} algorithm.


The data presented deal with the additive kernel $K(x,y) = \la (x+y)$ and a kernel that is used in modelling soot formation in a free molecular regime \cite{GoodKraFict,Wagner,PattersonSinghBalthasar}, thus we shall call it the `Soot Kernel':
\begin{equation*}
K(x,y) = \left(\frac{1}{x}+\frac{1}{y}\right)^{\frac{1}{2}} \left(x^{\frac{1}{\lambda}} + y^{\frac{1}{\lambda}}\right)^2;
\end{equation*}
both are considered in the \textit{discrete setting where masses are integers}. The reference value of $\la$ for the additive kernel will be $1$ and for the soot kernel $2.1$. We shall always take as initial condition for the Marcus-Lushnikov process $N$ particles with mass equal to $1$, and $\sigma^{+,N}_0 = \sigma^{+,N}_0 = 0$.\\


Smoluchowski equation has an explicit analytic solution for an additive interaction (see the review by Aldous \cite{AldousRev} for instance) we can compare our results with it; it will be convenient to write $\sigma^\infty_t$ for $\sigma_t$ in this case. No analytic solution of Smoluchowski equation or its sensitivity equation is available for the soot kernel; we shall thus compare our estimators $\sigma^N_t$ with what the ExactCoupling algorithm gives us for very high settings, say $N=3\times 10^6$ and $L=10^3$ simulations. Given any $N$, the $l^{\textrm{th}}$ run of the algorithm produces an estimator of $\sigma_t$ which we shall denote by $\sigma_t^{l,N}$. We shall set $\sigma^\infty_t := 10^{-3}\displaystyle{\sum_{l=1,\ldots,10^3}\sigma_t^{l,10^6}}$. Figures \ref{Additive:partsizeplot} and \ref{Soot:partsizeplot} show the empirical estimate of $\sigma_t$ given after $L$ runs, at different times. The line represents $\sigma_t^\infty$. For comparison, the results given by the CD algorithm for the same setting, with $\delta\la = 0.05$, are plotted using stars. Also, Figure \ref{solution_plots} shows what the solution to the original Smoluchowski equation looks like.

\begin{figure}[t]
\begin{center}
    \subfigure[!h][ExactCoupling, $t=0.5$]
    {
        \resizebox{!}{0.27\linewidth}{\includegraphics{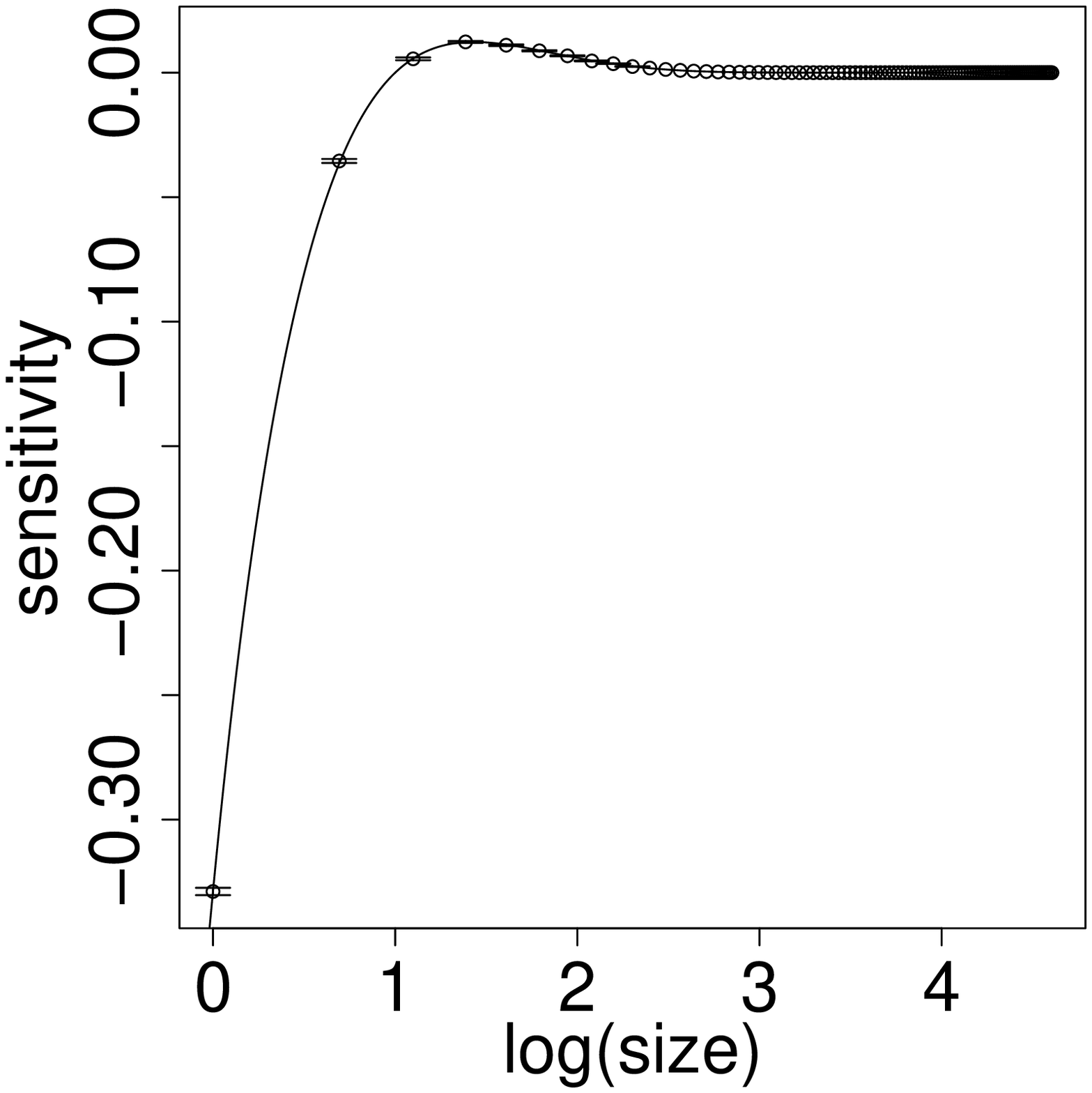}}
    }\hspace{0.5cm}
    \subfigure[!h][CD ($\delta\lambda=0.05$), $t=0.5$]
    {
        \resizebox{!}{0.27\linewidth}{\includegraphics{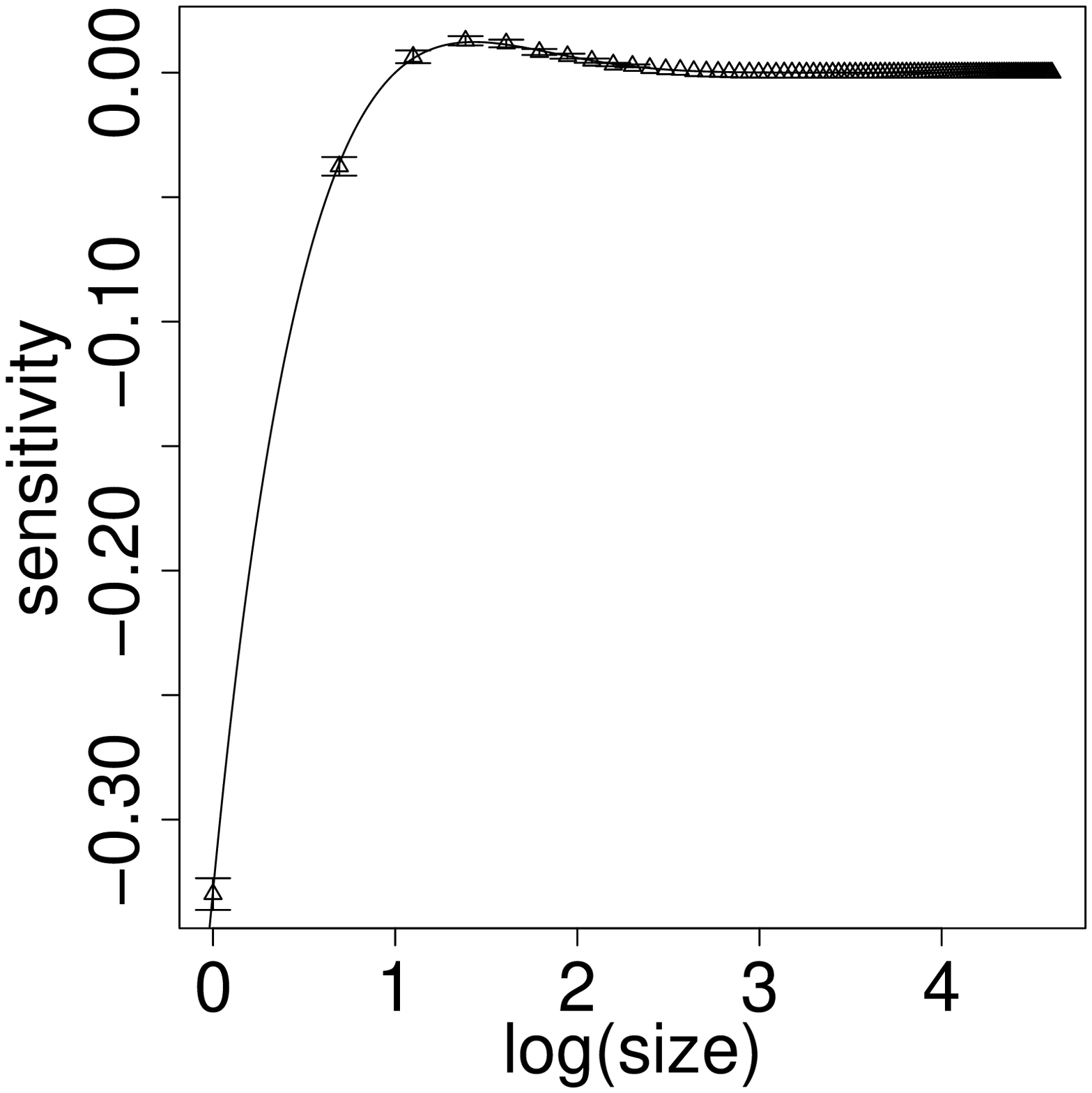}}
    }\\
    \subfigure[!h][ExactCoupling, $t=3.0$]
    {
        \resizebox{!}{0.27\linewidth}{\includegraphics{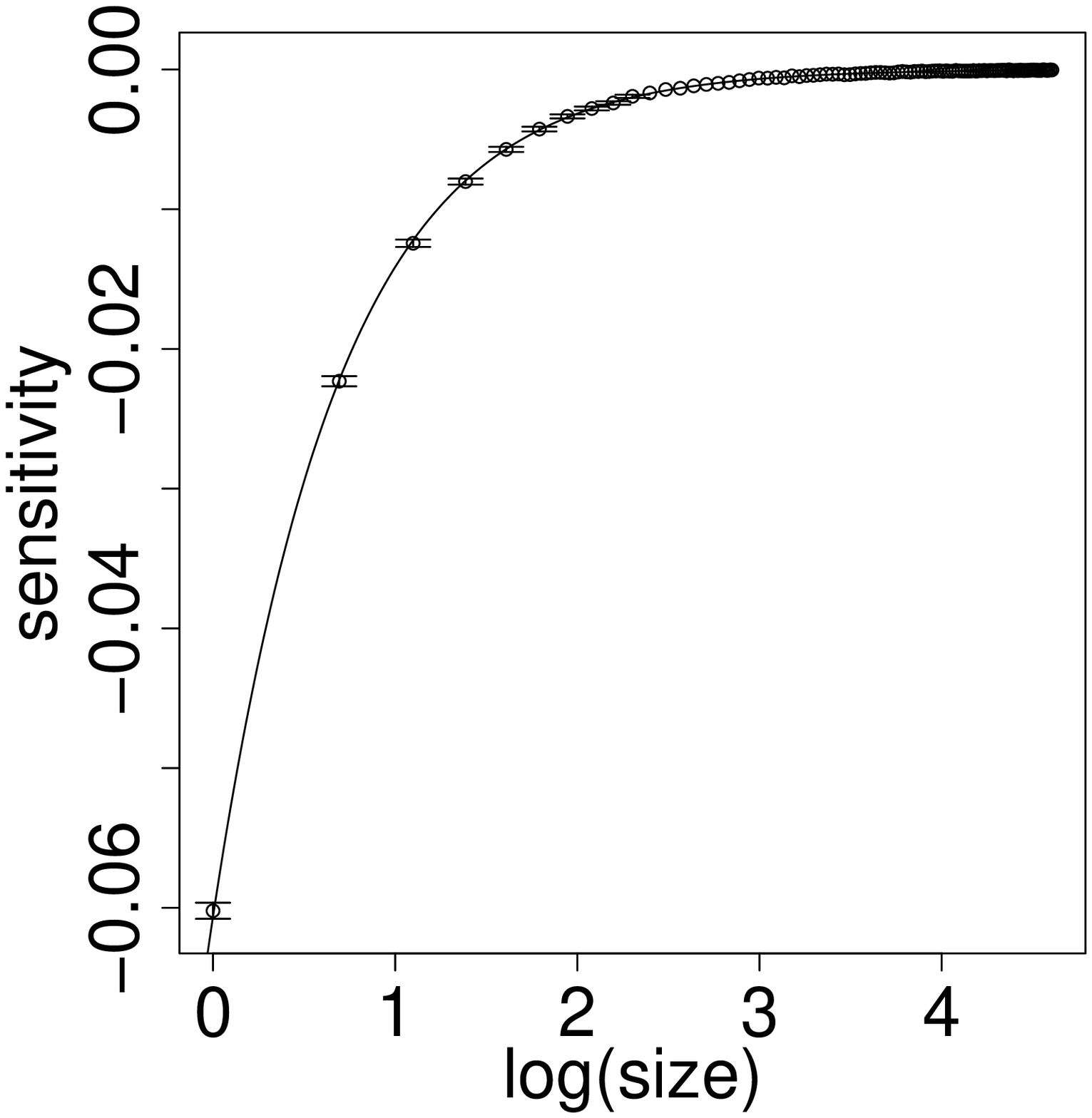}}
    }\hspace{0.5cm}
    \subfigure[!h][CD ($\delta\lambda=0.05$), $t=3.0$]
    {
        \resizebox{!}{0.27\linewidth}{\includegraphics{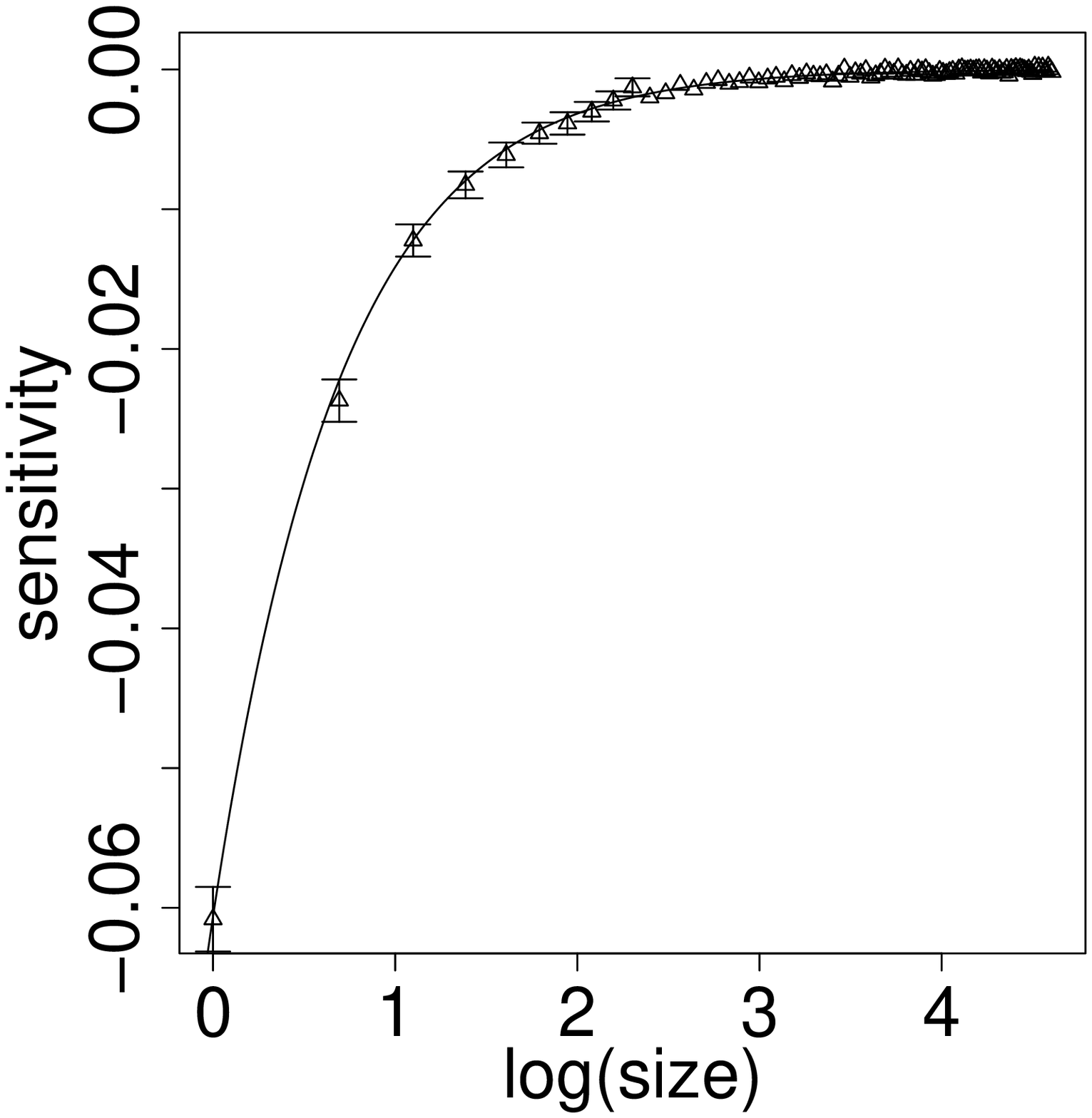}}
    }\\
\end{center}
\caption{Sensitivity for additive kernel, $\lambda=1.0$, $N=10^3$, $L=1000$. The confidence intervals for the larger particle sizes are omitted for pictorial clarity.}
\label{Additive:partsizeplot}
\end{figure}
\begin{figure}[t]
\begin{center}
    \subfigure[!h][ExactCoupling, $t=0.5$]
    {
        \resizebox{!}{0.27\linewidth}{\includegraphics{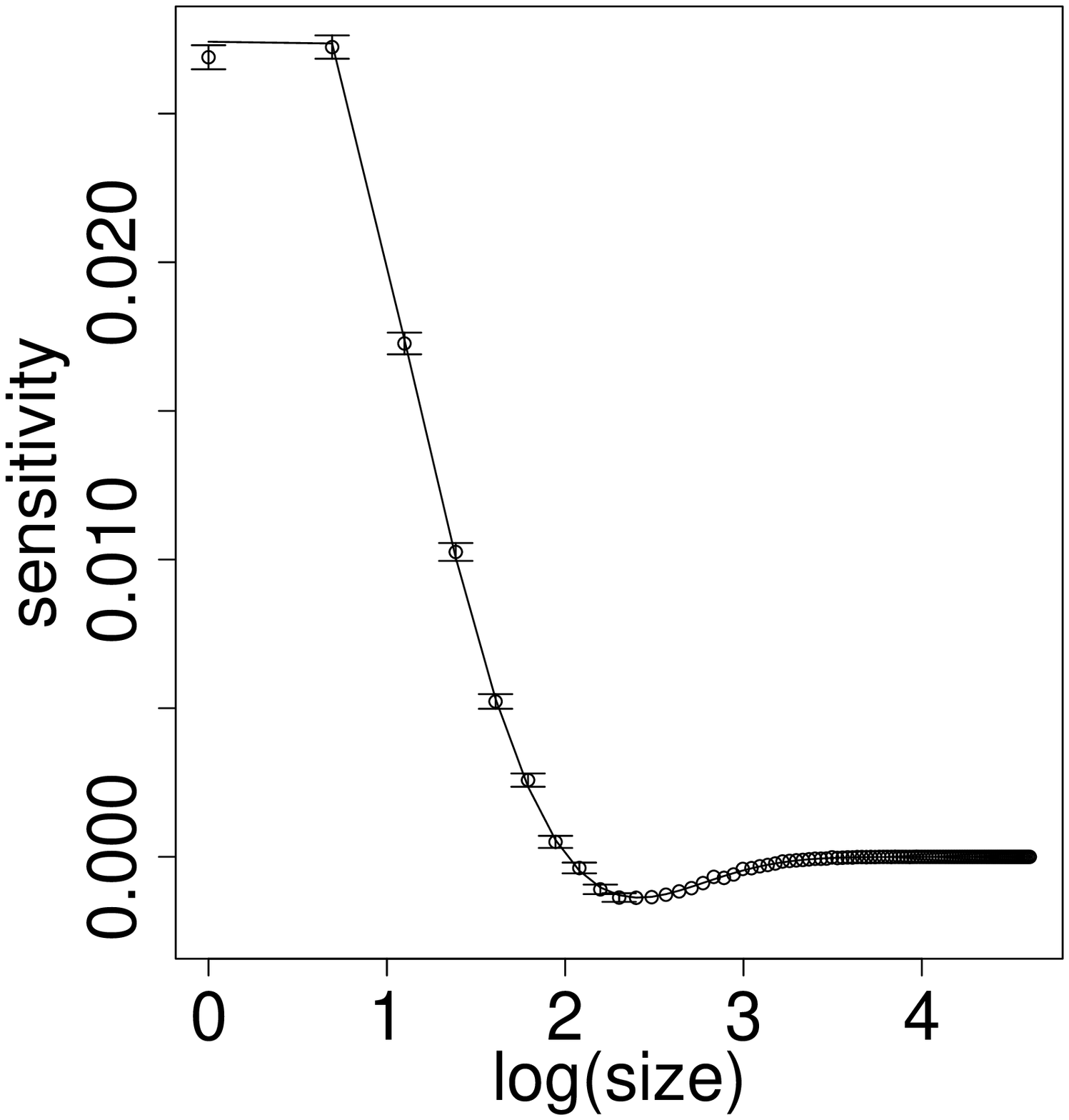}}
    }\hspace{0.5cm}
    \subfigure[!h][CD ($\delta\lambda=0.05$), $t=0.5$]
    {
        \resizebox{!}{0.27\linewidth}{\includegraphics{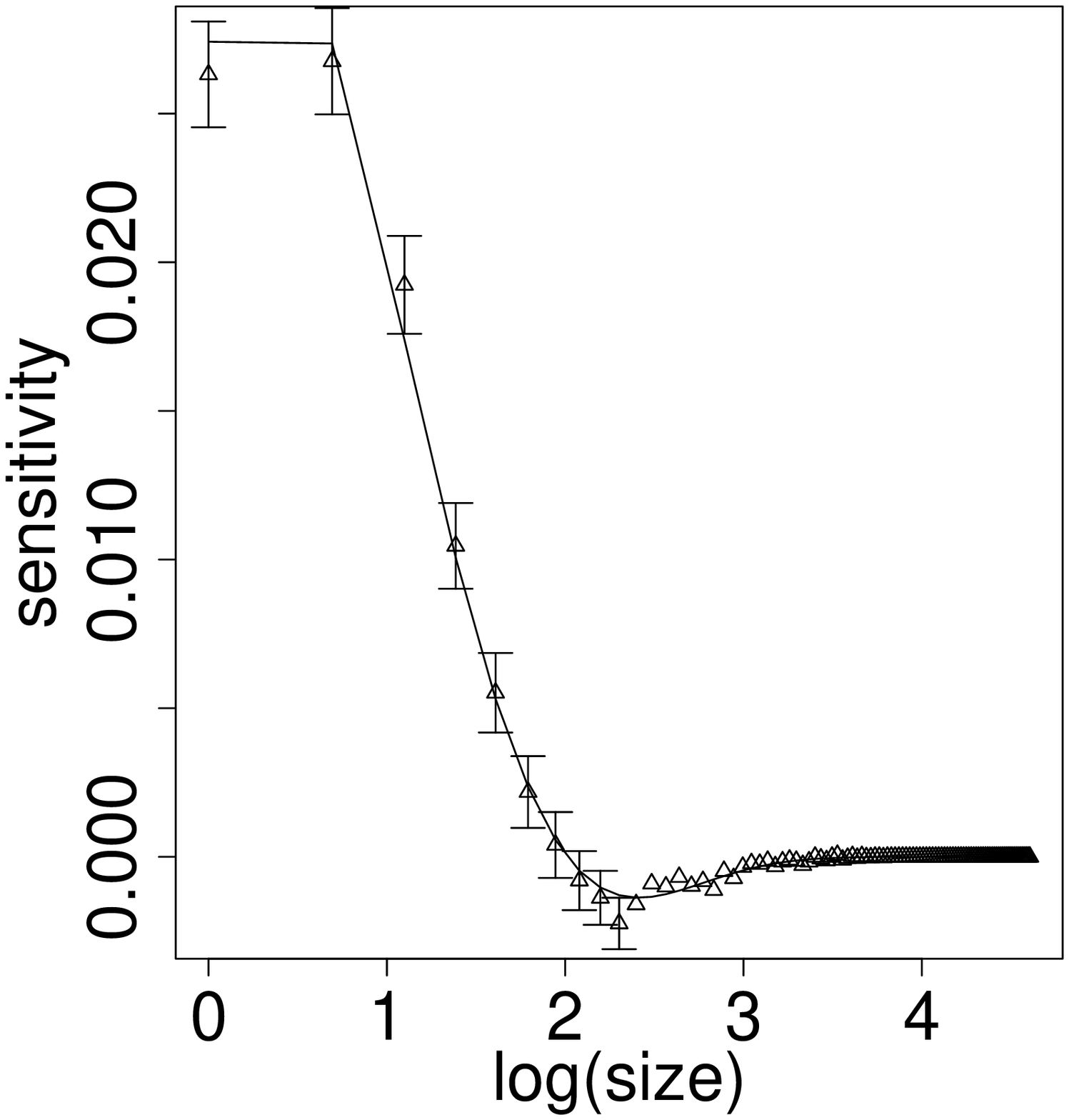}}
    }\\
    \subfigure[!h][ExactCoupling, $t=3.0$]
    {
        \resizebox{!}{0.27\linewidth}{\includegraphics{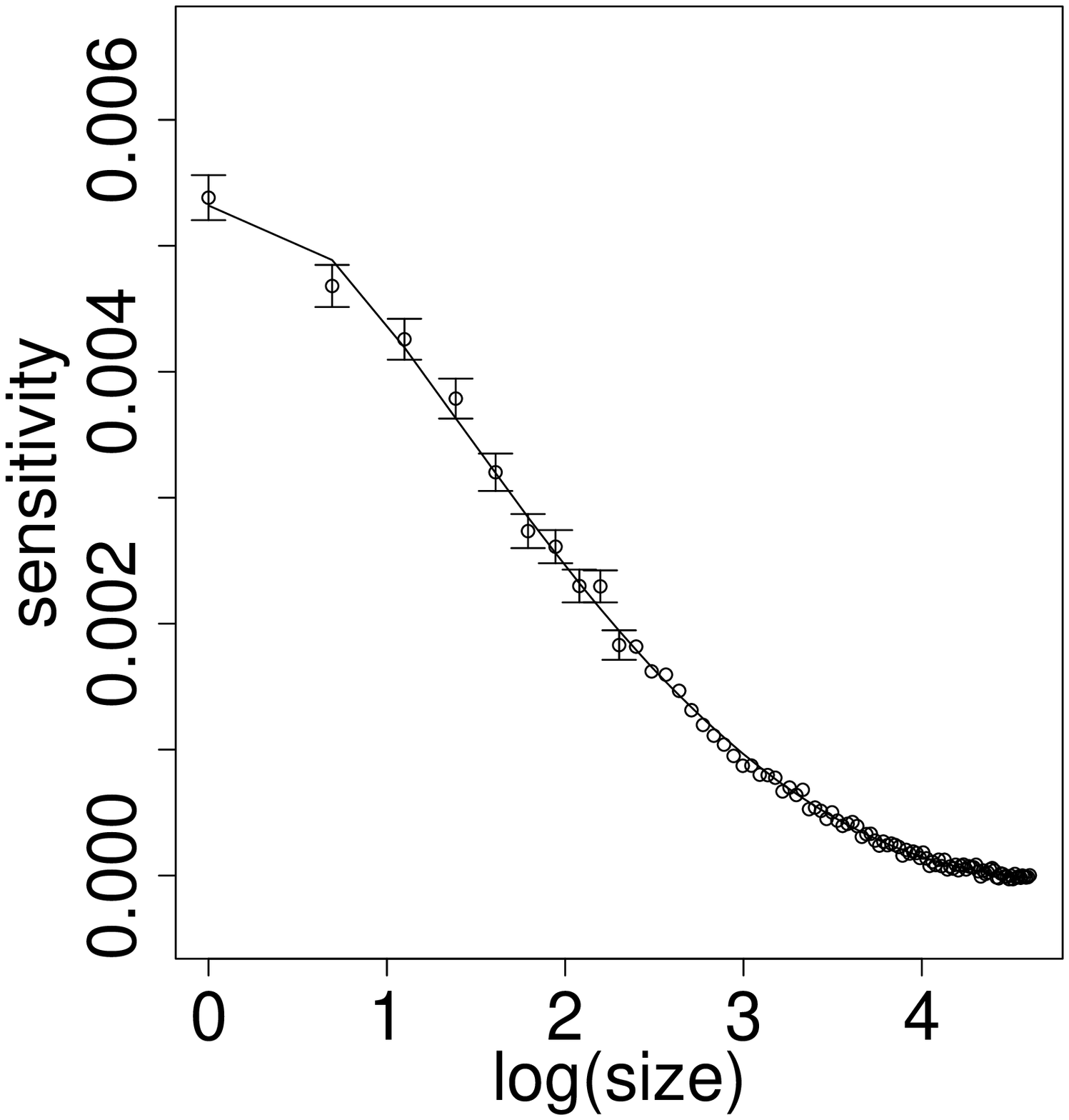}}
    }\hspace{0.5cm}
    \subfigure[!h][CD ($\delta\lambda=0.05$), $t=3.0$]
    {
        \resizebox{!}{0.27\linewidth}{\includegraphics{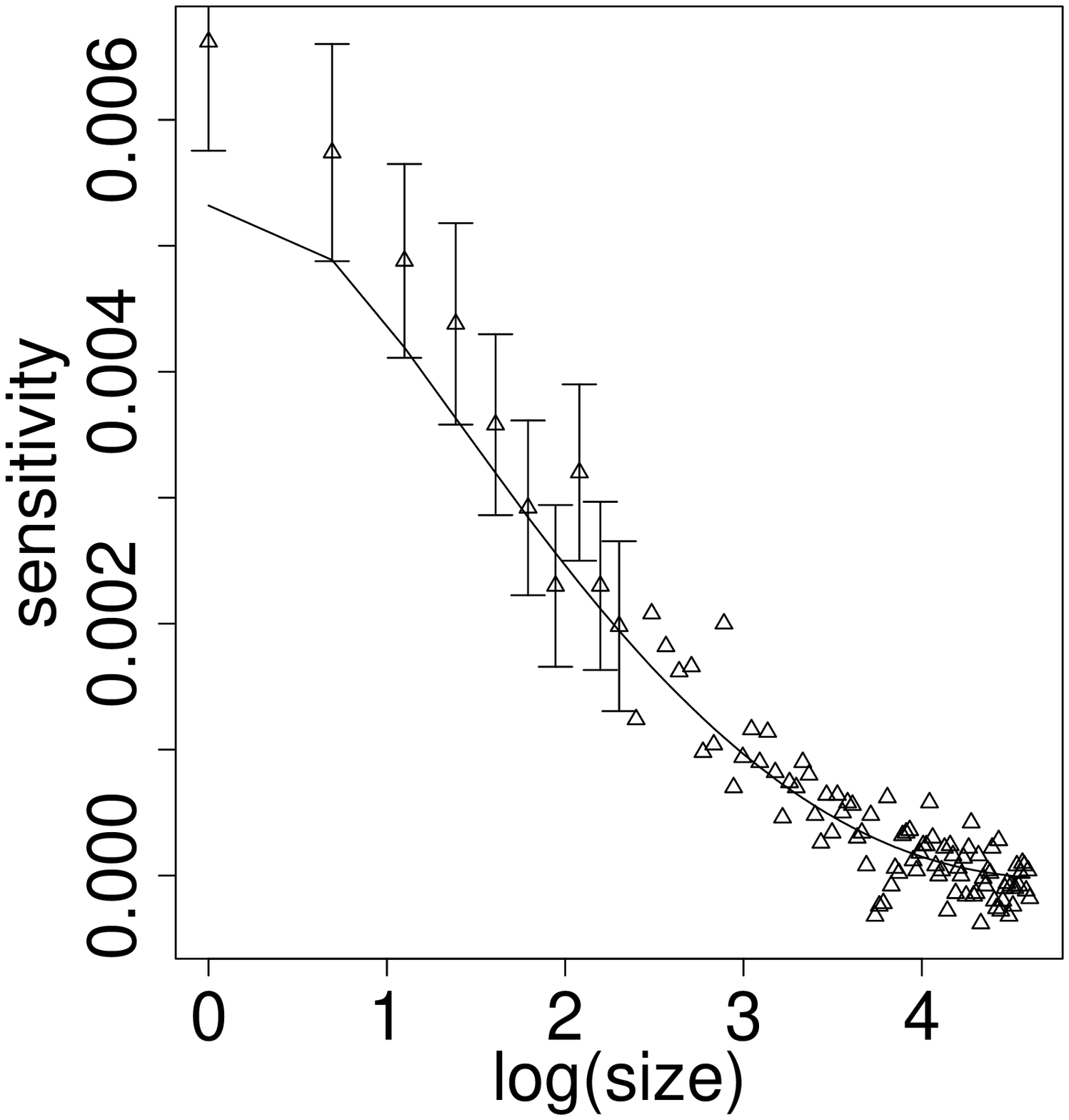}}
    }
 \end{center}
\caption{Sensitivity for soot kernel, $\lambda=2.1$, $N=10^3$, $L=1000$. The confidence intervals for the larger particle sizes are omitted for pictorial clarity.}
\label{Soot:partsizeplot}
\end{figure}

\begin{figure}[!h]
\begin{center}
    \subfigure[h][Additive kernel, $t=1.0$]
    {
        \resizebox{!}{0.30\linewidth}{\includegraphics{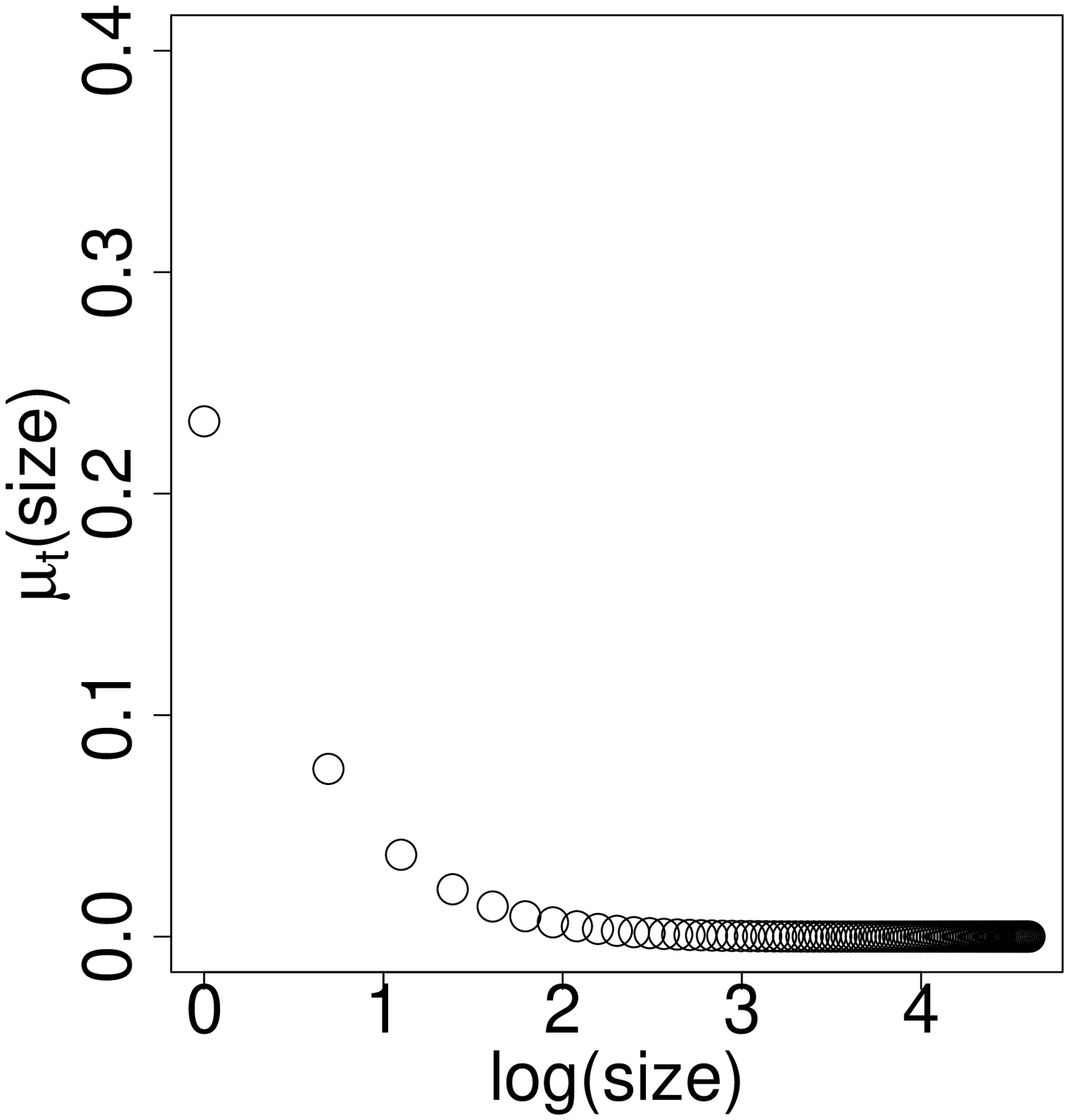}}
    }
    \subfigure[h][Additive kernel, $t=3.0$]
    {
        \resizebox{!}{0.30\linewidth}{\includegraphics{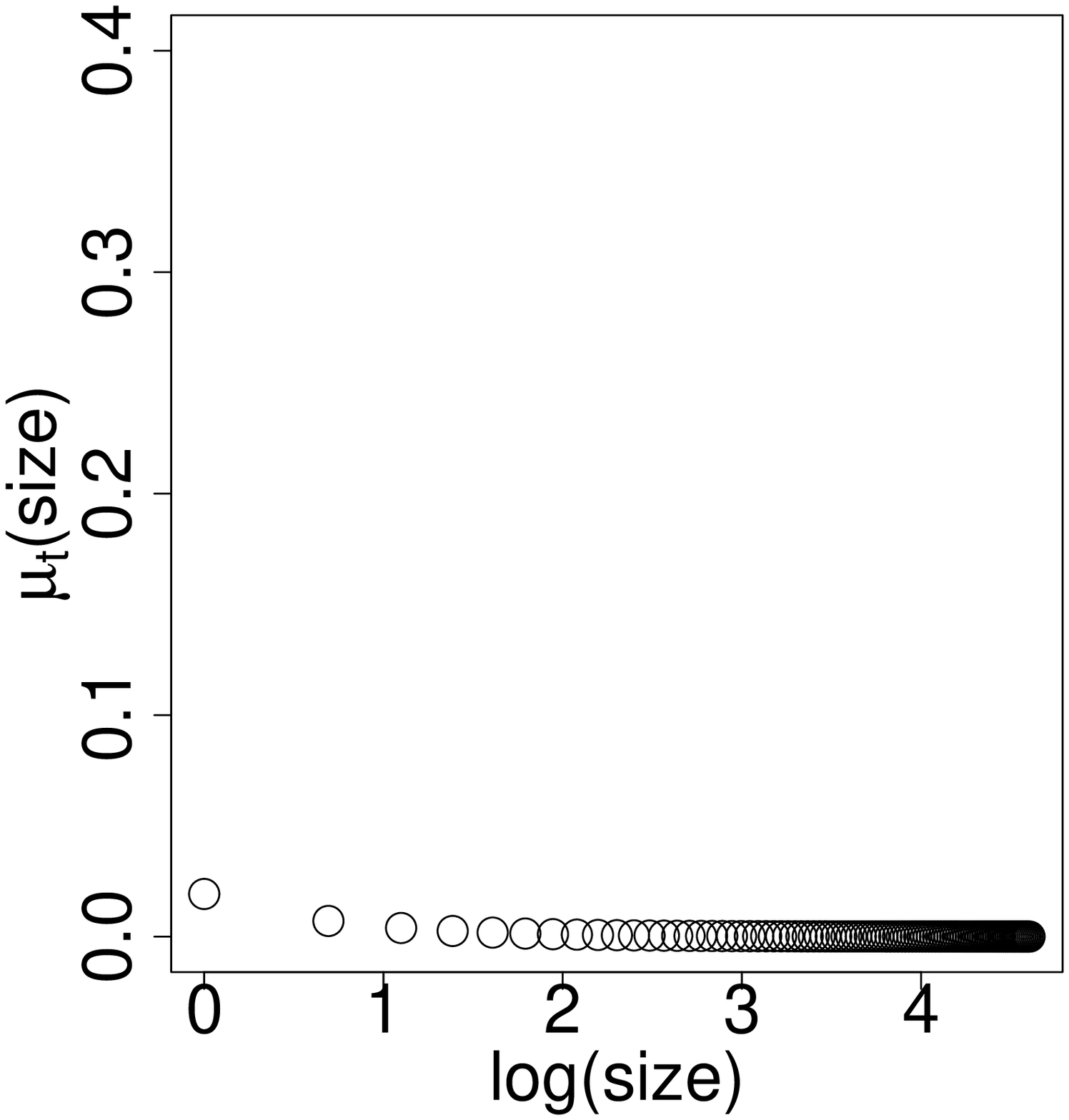}}
    }\\
    \subfigure[h][Soot kernel, $t=1.0$]
    {
        \resizebox{!}{0.30\linewidth}{\includegraphics{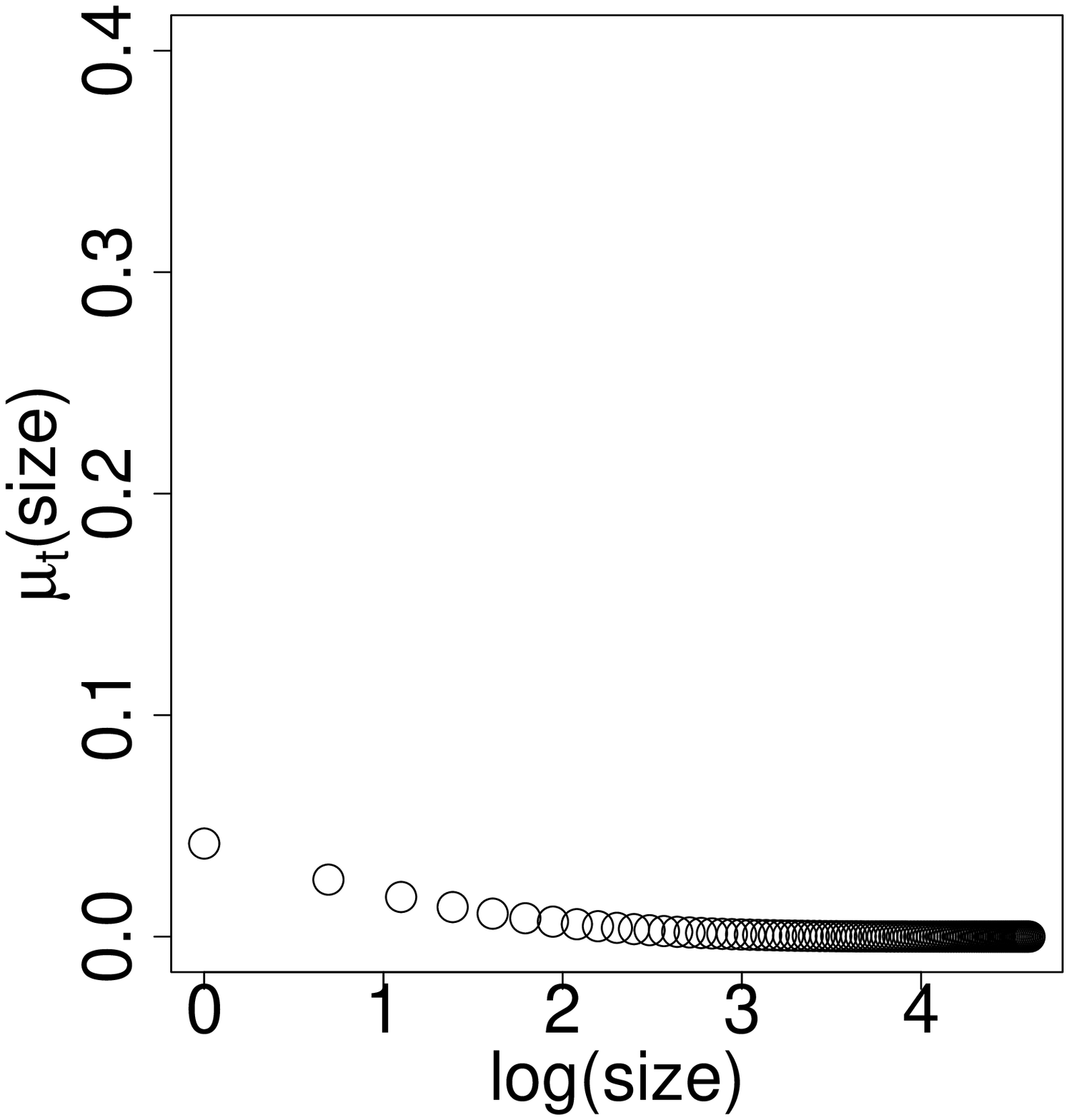}}
    }
    \subfigure[h][Soot kernel, $t=3.0$]
    {
        \resizebox{!}{0.30\linewidth}{\includegraphics{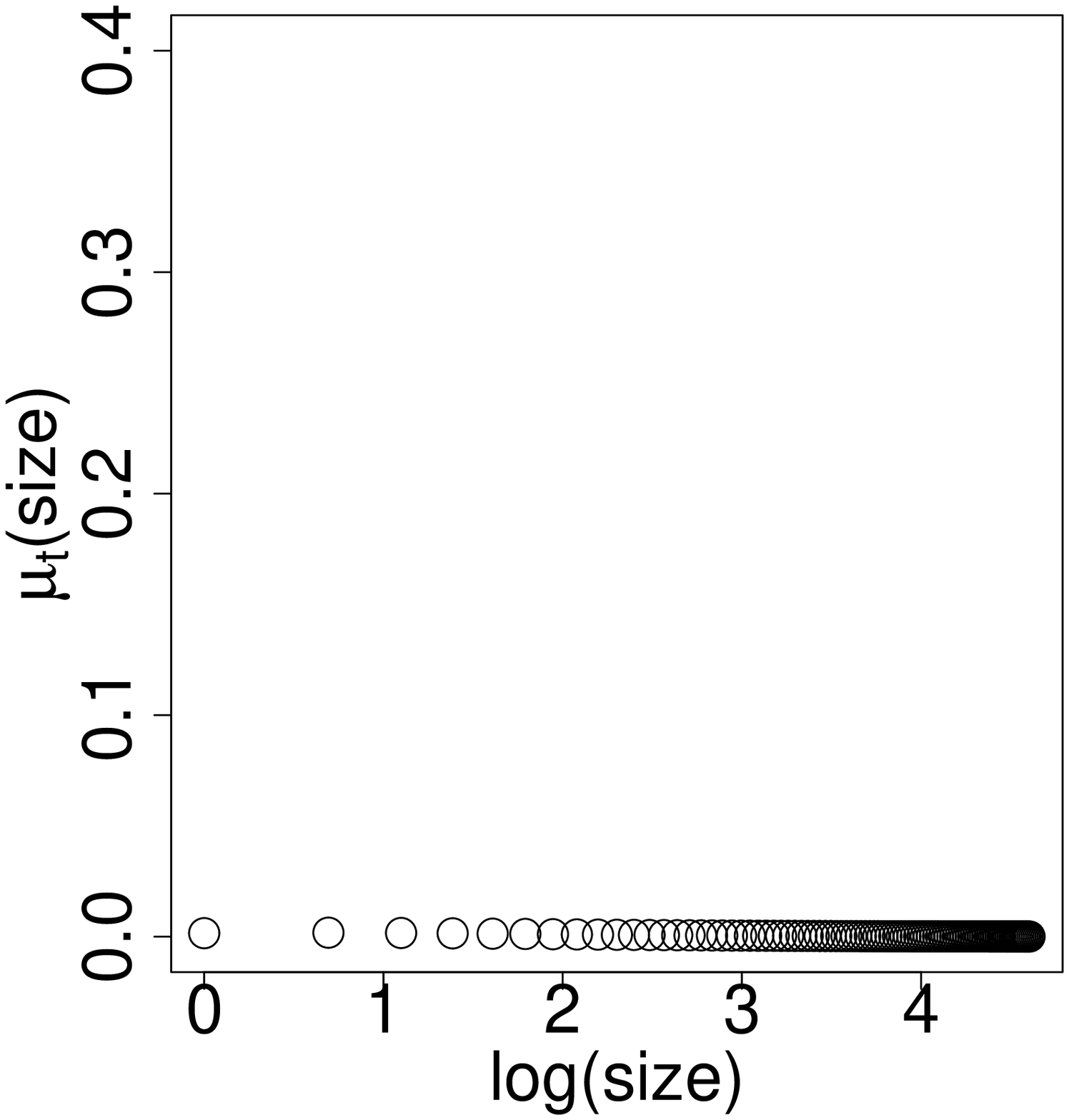}}
    }
 \end{center}
   \caption{$\mu_t$ as a function of $\log(\textrm{particle size})$}
   \label{solution_plots}
\end{figure}

To quantify the convergence of the empirical sensitivity
$$
\bar{\sigma}^{L\,;\,N}_t:= \frac{1}{L}\sum_{l=1..L}\sigma^{l,N}_t
$$
to $\sigma^\infty_t$ as $N$ increases we have plotted in Figure \ref{Convergence_study} the quantity
$$
d_{\textrm{var}}(N) = \sum_j \sum_{i\geq 1} \Big|\bigl(\bar{\sigma}^{L\,;\,N}_{t_j} - \sigma^\infty_{t_j}\bigr)(i)\Bigr|,
$$
where $\bar{\sigma_t}^{L\,;\,N}(i)$ and $\sigma_t^{\infty}(i)$ represent the empirical and real sensitivities at particle mass $i \in \mathbb{N}$ respectively, and $d_{\textrm{var}}(N)$ represents the total variation distance between the empirical sensitivity and the sensitivity itself summed over some chosen time points\footnote{For Figure \ref{Convergence_study}, the times points $\{t_j\}$ were chosen to be $0.125j$ for $j=1,\ldots,56$} $\{t_j\}$. These results empirically confirm Theorem \ref{ThmConvergence} (in this case where it does not apply), and quantify the speed of convergence as being of order $\frac{1}{N}$. The analogue result for the CD algorithm is given in \cite{PeterJamesMarkus}.
\begin{figure}[!h]
\begin{center}
    \subfigure[][Additive kernel]
    {
        \resizebox{!}{0.33\linewidth}{\includegraphics{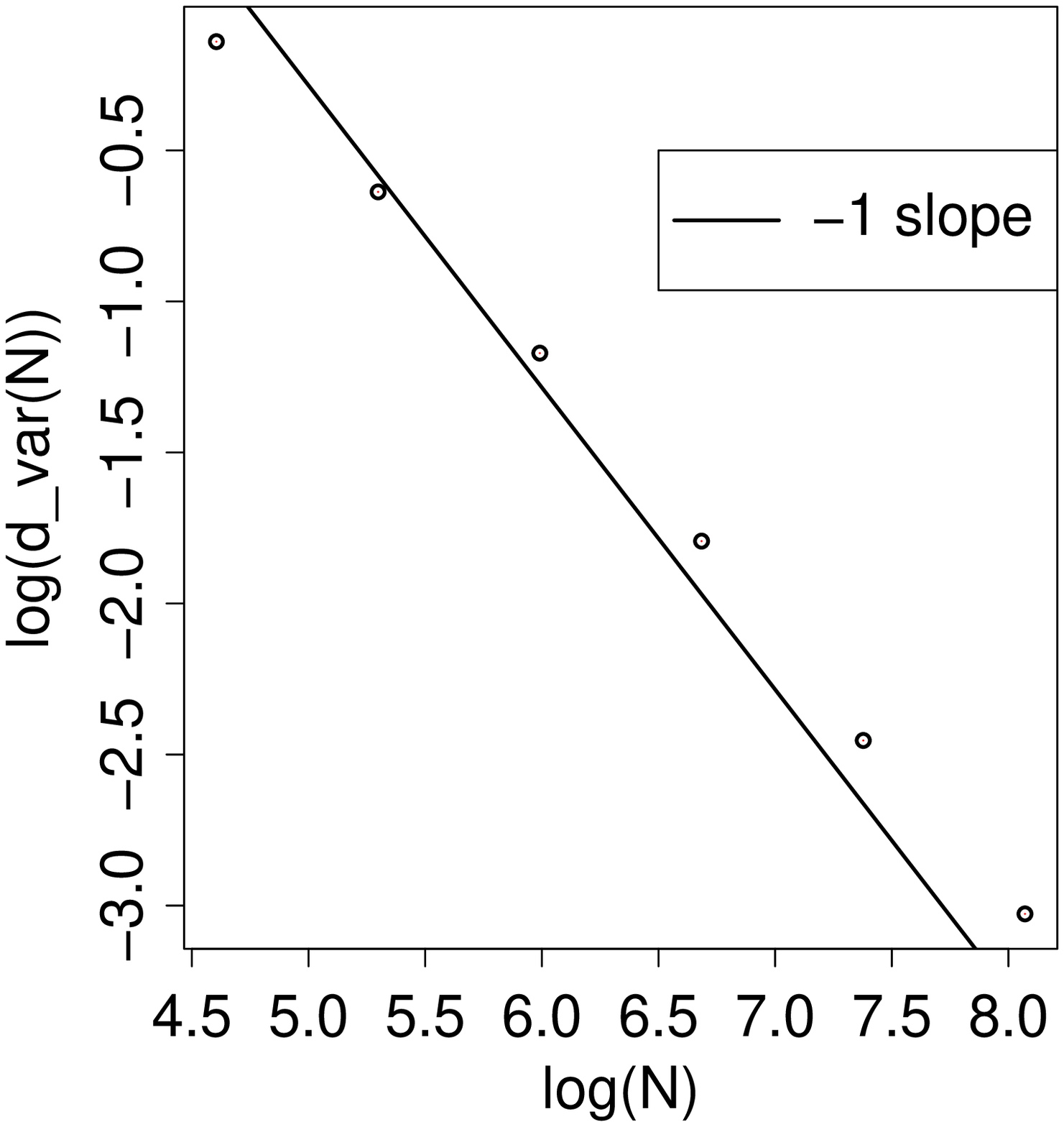}}
    }
    \subfigure[][Soot kernel]
    {
        \resizebox{!}{0.33\linewidth}{\includegraphics{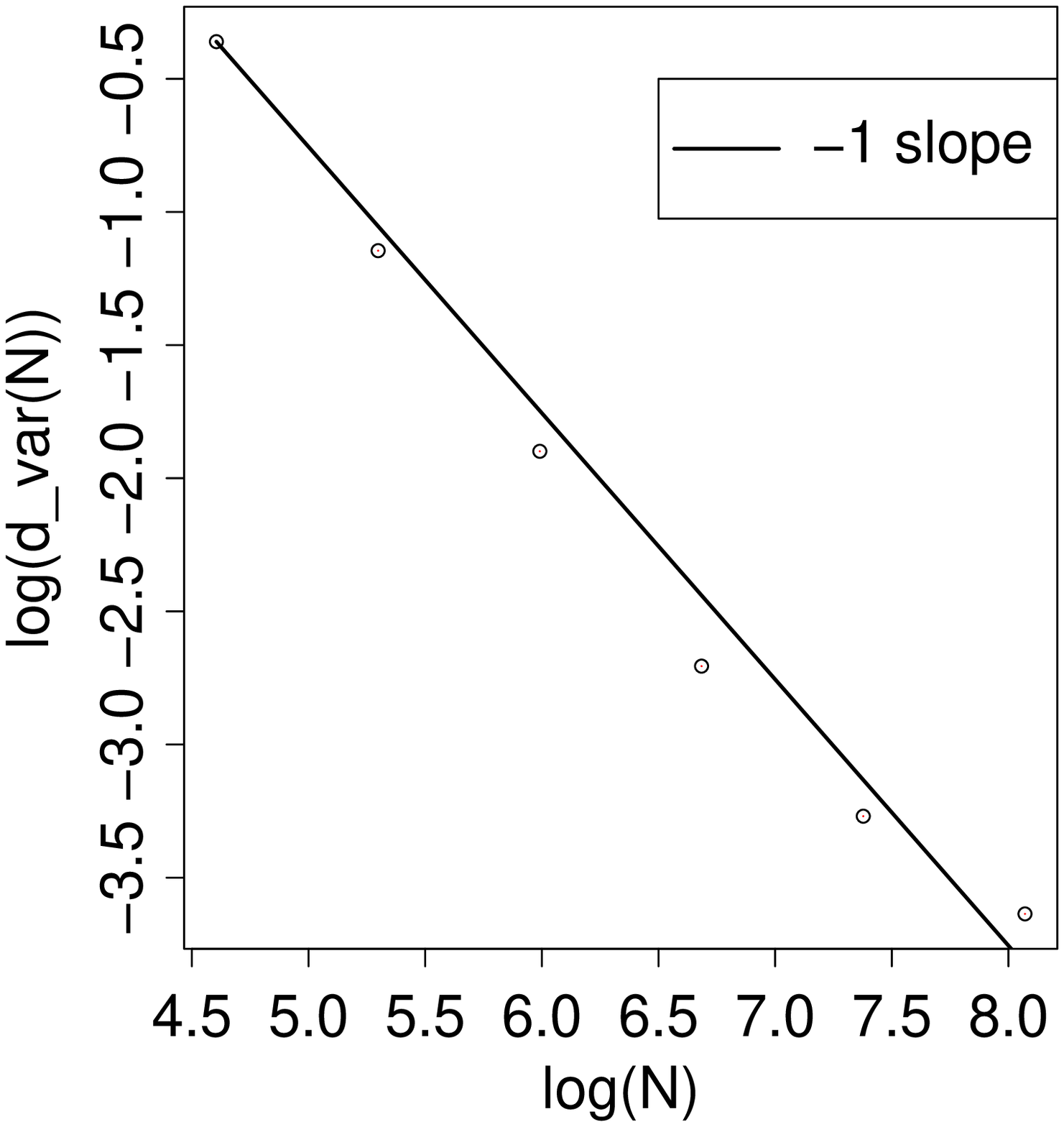}}
    }
    \caption{Convergence in $N$ of the ExactCoupling algorithm, $N=100 \times 2^i$ for $i=0,\ldots,5$, $NL=2\times10^8$.}
    \label{Convergence_study}
\end{center}
\end{figure}

\begin{figure}[!h]
\begin{center}
    \subfigure[][Additive kernel, $t=1.0$]
    {
        \resizebox{!}{0.3\linewidth}{\includegraphics{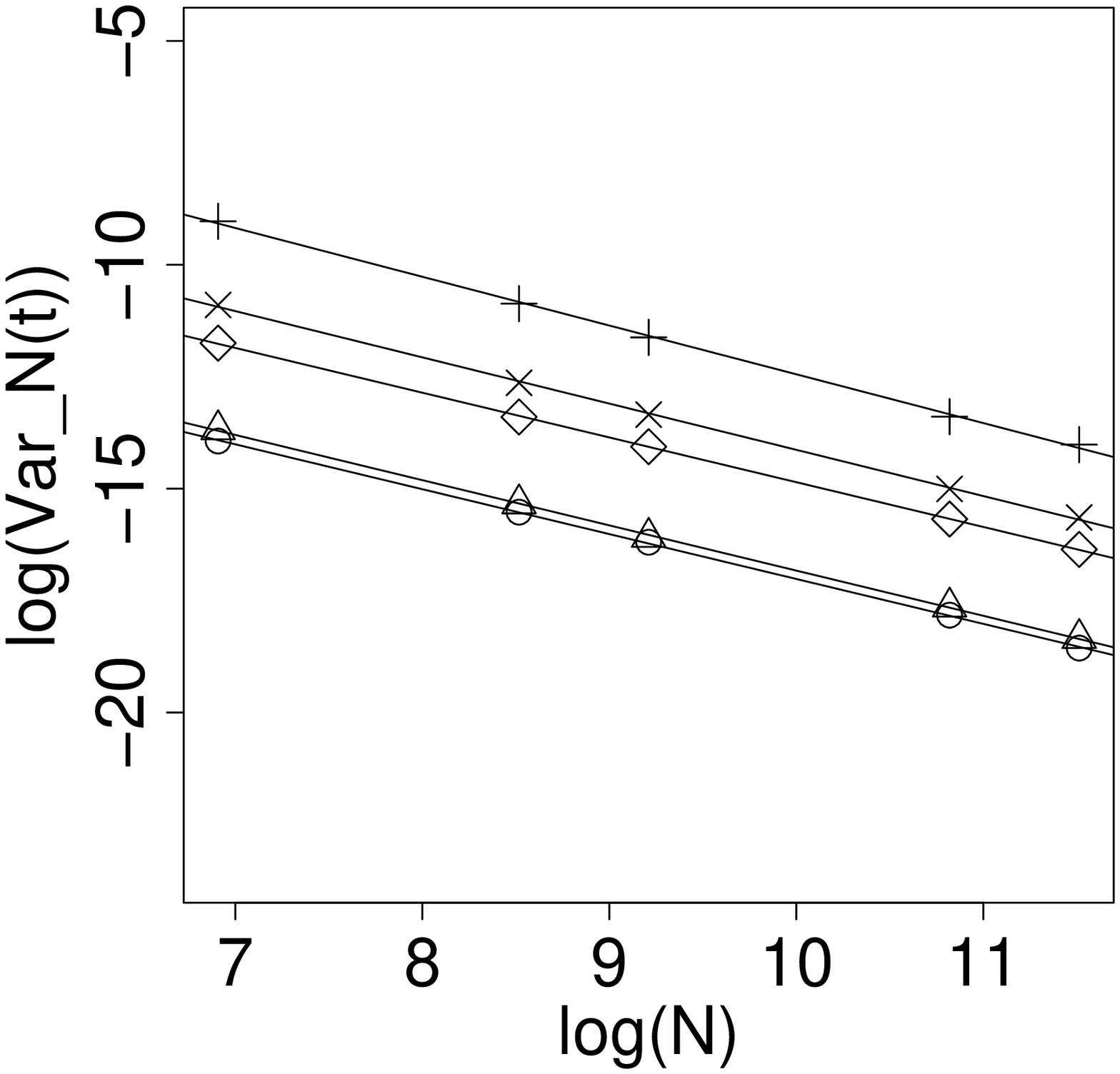}}
    }
    \subfigure[][Additive kernel, $t=5.0$]
    {
        \resizebox{!}{0.3\linewidth}{\includegraphics{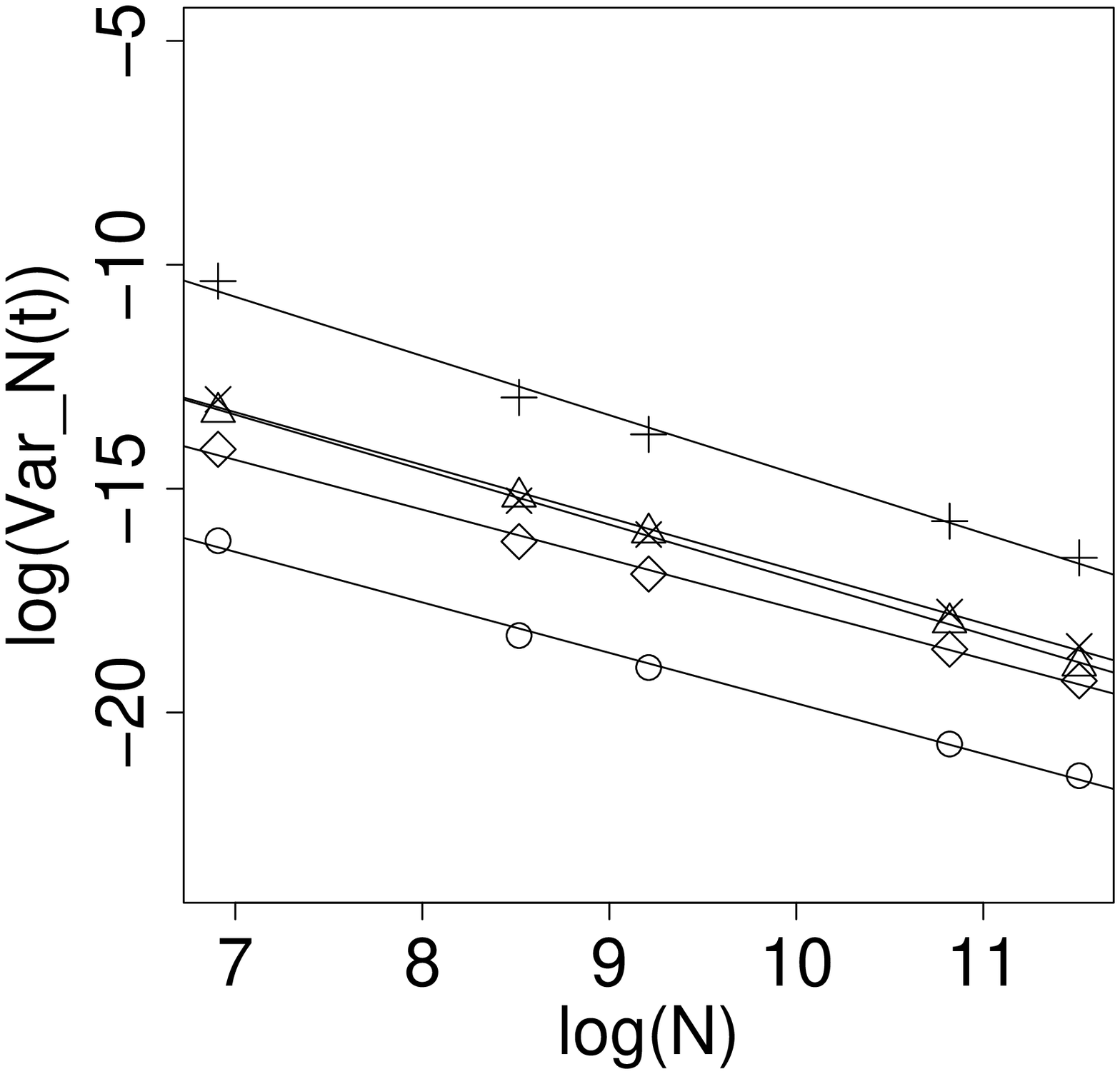}}
    }
\end{center}
\begin{center}
    \subfigure[][Soot kernel, $t=1.0$]
    {
        \resizebox{!}{0.3\linewidth}{\includegraphics{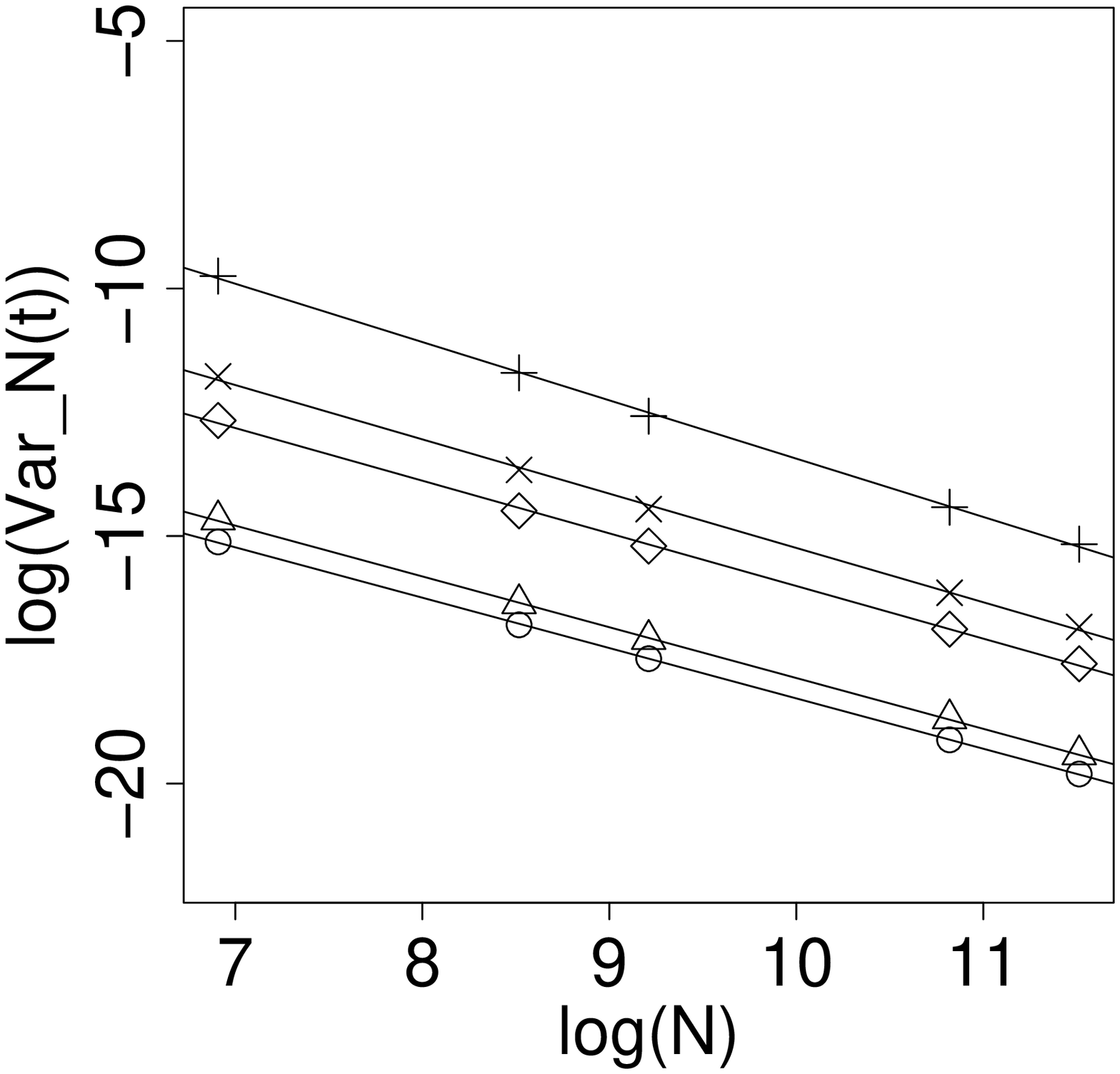}}
    }
    \subfigure[][Soot kernel, $t=5.0$]
    {
        \resizebox{!}{0.3\linewidth}{\includegraphics{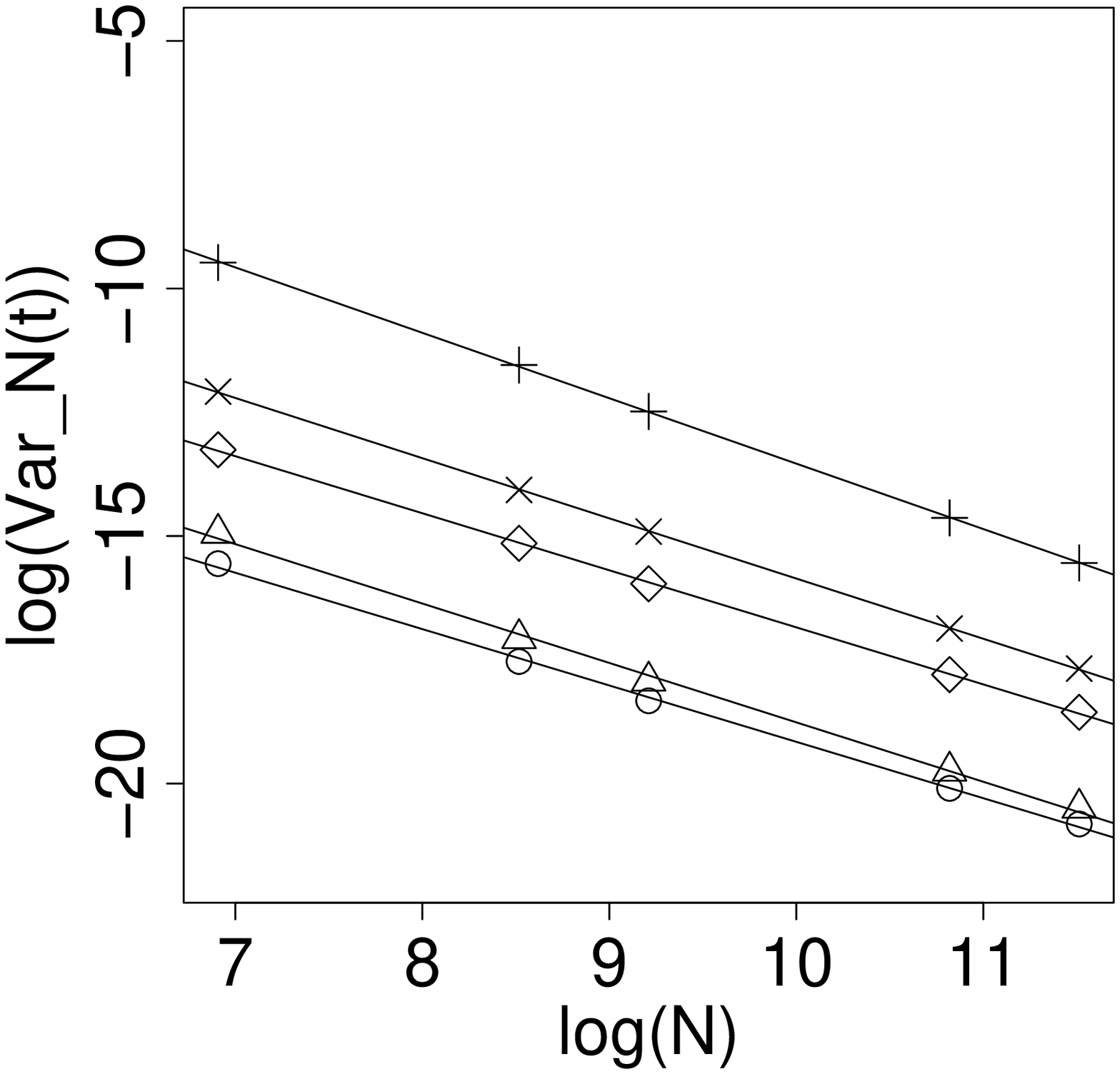}}
    }
    \caption{$\log \Var_N(t)$ as a function of $N$. The meaning of the symbols are as follows: Circles = ExactCoupling, Diamonds = ExactIndep, Triangles = CD($\delta\la=0.10$), Crosses = CD($\delta\la=0.05$), Pluses = CD($\delta\la=0.01$).}
    \label{Additive_statistical_error_versus_N}
\end{center}
\end{figure}

\subsection{Variance}
\label{Variance}
To analyse the variance of the random output of the algorithm we shall define the empirical variance at particle mass $i \in \mathbb{N}$ and time $t$ as
\begin{equation*}
\Var_N(i,t) := \frac{1}{L-1} \sum_{l=1}^L ( (\sigma_t^{l,N} - \bar{\sigma}_t^{L\,;\,N})(i) )^2
\end{equation*}
and shall take as a measure of the variance the quantity
\begin{equation}
\label{EmpiricalVariance}
\Var_N(t) := \sum_{i\geq 1} \Var_N(i,t).
\end{equation}
Figure \ref{Additive_statistical_error_versus_N} represents its graph as a function of $N$ using different algorithms. It shows that the ExactCoupling algorithm achieves a variance reduction by a factor $10^3$ compared to the CD algorithm. The plots also show that $\Var_N(t)$ is proportional to $\frac{1}{N}$, a fact that should be related to a central limit theorem.

\subsection{Computational efficiency}

Although section \ref{Variance} indicates that the ExactCoupling algorithm produces very accurate estimators of the sensitivity, it comes at the price of a computational time greater than the one needed by the CD algorithm. This comes from the fact that the latter algorithm being essentially a Marcus-Lushnikov algorithm, it uses a generally decreasing amount of information, as the number of sensitivity particles decreases with time. On the other hand, the ExactCoupling algorithm has to deal with more and more sensitivity particles, whose number tends to grow exponentially. To see whether the gain of accuracy given by the ExactCoupling algorithm is worth the effort we propose two criteria.

\subsubsection{CPU time to reach a certain level of accuracy}
\label{CPU_time_subsubsection}

Fix the observation time $t$ (we choose large enough $t$ so that the particle system has experienced many jumps, and therefore the variances are expected to be larger - see Figure \ref{solution_plots}. Given a certain level of accuracy, $v$, find for each algorithm the smallest $N$ for which $\Var_N(t)$ is smaller than $v$. See what computational time is needed to run the algorithm for this $N$ (during an evolution time $t$ for the particle system). Tables \ref{Additive_Smallest_N} and \ref{Soot_Smallest_N} show that the ExactCoupling algorithm remains mostly better than the CD algorithm. It also shows that it converges much quicker to the true sensitivity than the CD algorithm does. Note that for the soot kernel the CD algorithm with $\delta\la = 0.1$, $10^5$ initial particles are not sufficient to reach the given level of accuracy; this setup already requires a CPU time equal to $1058.91$ seconds. The comparison with the corresponding time for the ExactCoupling algorithm is greatly in favour of the latter.

\begin{table}[!h]
\caption{Additive kernel, $v = 1.43 \times 10^{-4}$}
\label{Additive_Smallest_N}
\begin{center}
\begin{tabular}{l|ll|ll}
$t$ & 1.0 & 1.0 & 3.0 & 3.0\\
\hline
algorithm & ExactCoupling & CD ($\delta\la=0.10$) & ExactCoupling & CD ($\delta\la=0.10$)\\
$N$ & 6500 & 55000 & 2100 & 16250\\
$t_{\textrm{run}}$ (secs) & 281.15 & 593.99 & 99.22 & 213.34\\
\end{tabular}
\end{center}
\end{table}
\begin{table}[!h]
\caption{Soot kernel, $v = 2.57 \times 10^{-5}$}
\label{Soot_Smallest_N}
\begin{center}
\begin{tabular}{l|ll|ll}
$t$ & 1.0 & 1.0 & 3.0 & 3.0\\
\hline
algorithm & ExactCoupling & CD ($\delta\la=0.10$) & ExactCoupling & CD ($\delta\la=0.10$)\\
$N$ & 10000 & 100000 & 6350 & 55000\\
$t_{\textrm{run}}$ (secs) & 379.01 & 1058.91 & 382.15 & 1104.24\\
&&($v$ not reached) &&
\end{tabular}
\end{center}
\end{table}

\subsubsection{Gain factor}
\label{Gain_factor}

Eibeck and Wagner introduced in \cite{Wagner3} another quantity to compare the relative efficiency of two algorithms. Fix the observation time $t$. Given a setup ($K(\cdot,\cdot), N, L$), denote by $T^{\textrm{EC}}(t)$ and $T^{\textrm{CD}}(t)$ the empirical mean CPU time needed by the ExactCoupling and CD algorithms to be run up to time $t$. Denote also by $\Var_N^{\textrm{EC}}(t)$ and $\Var_N^{\textrm{alg}}(t)$ the empirical variances given by formula \eqref{EmpiricalVariance} when computed using ExactCoupling and the given algorithm `alg' respectively. The \textit{gain factor} of an algorithm \textit{over ExactCoupling}, similar to that as introduced by Eibeck and Wagner, is defined here by the ratio
$$
\frac{T^{\textrm{EC}}(t)\,\Var_N^{\textrm{EC}}(t)}{T^{\textrm{alg}}(t)\,\Var_N^{\textrm{alg}}(t)}
$$
It is related in some way to the analysis made in section \ref{CPU_time_subsubsection}. See section $5$ of \cite{Wagner3}. Figures \ref{Add_inefficiency} and \ref{Soot_inefficiency} plot the reciprocal gain (its logarithm) as a function of time. Triangles, pluses and crosses represent data of the CD algorithm, for $\delta\la=0.01,\,0.05$ and $0.10$ respectively, circles represent data of the ExactIndep algorithm, and the horizontal line at zero represents the threshold for ExactCoupling.\\

\begin{figure}[!h]
    \subfigure[][$N=10^3$]
    {
        \resizebox{!}{0.3\linewidth}{\includegraphics{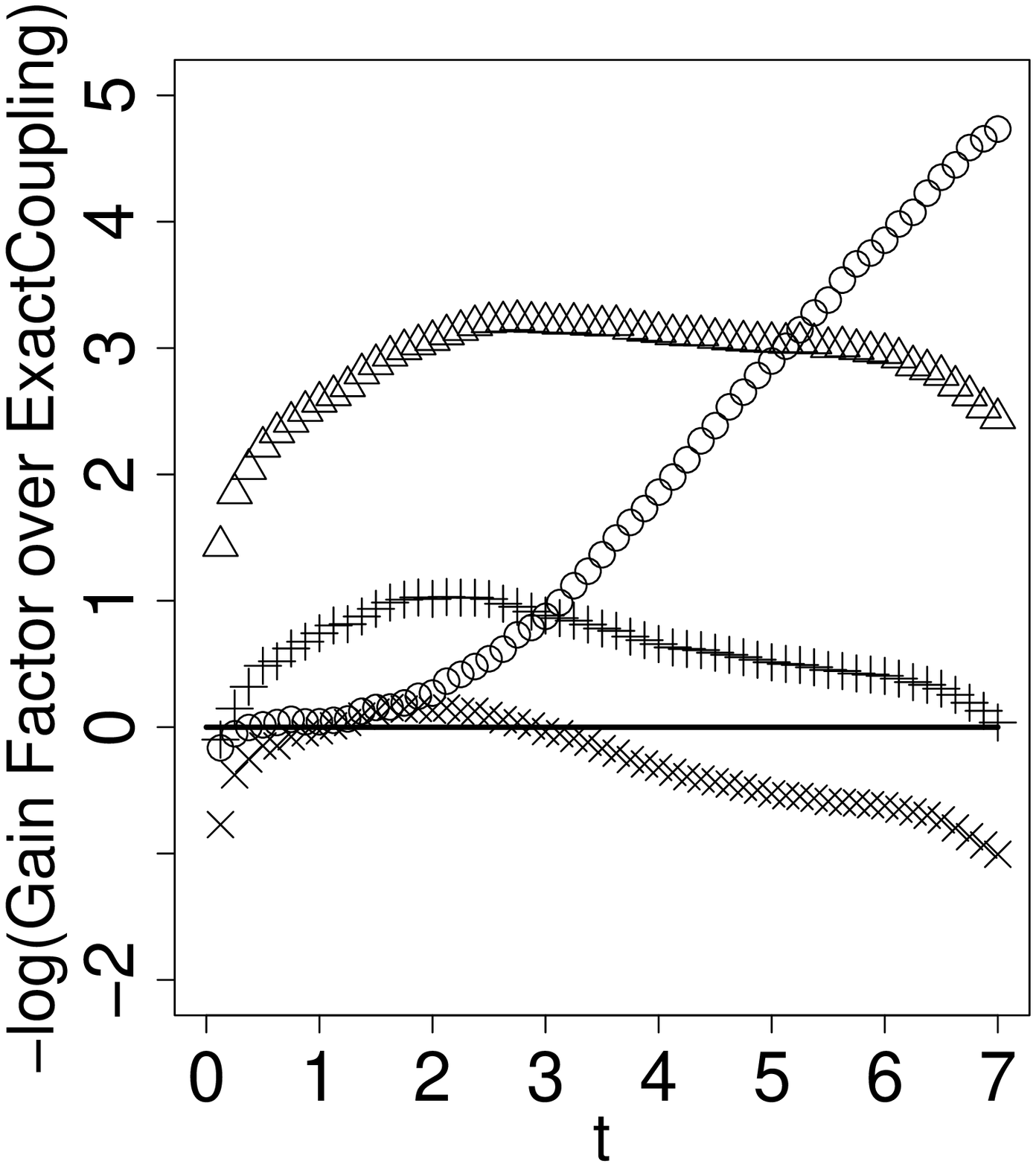}}
    }
    \subfigure[][$N=10^4$]
    {
        \resizebox{!}{0.3\linewidth}{\includegraphics{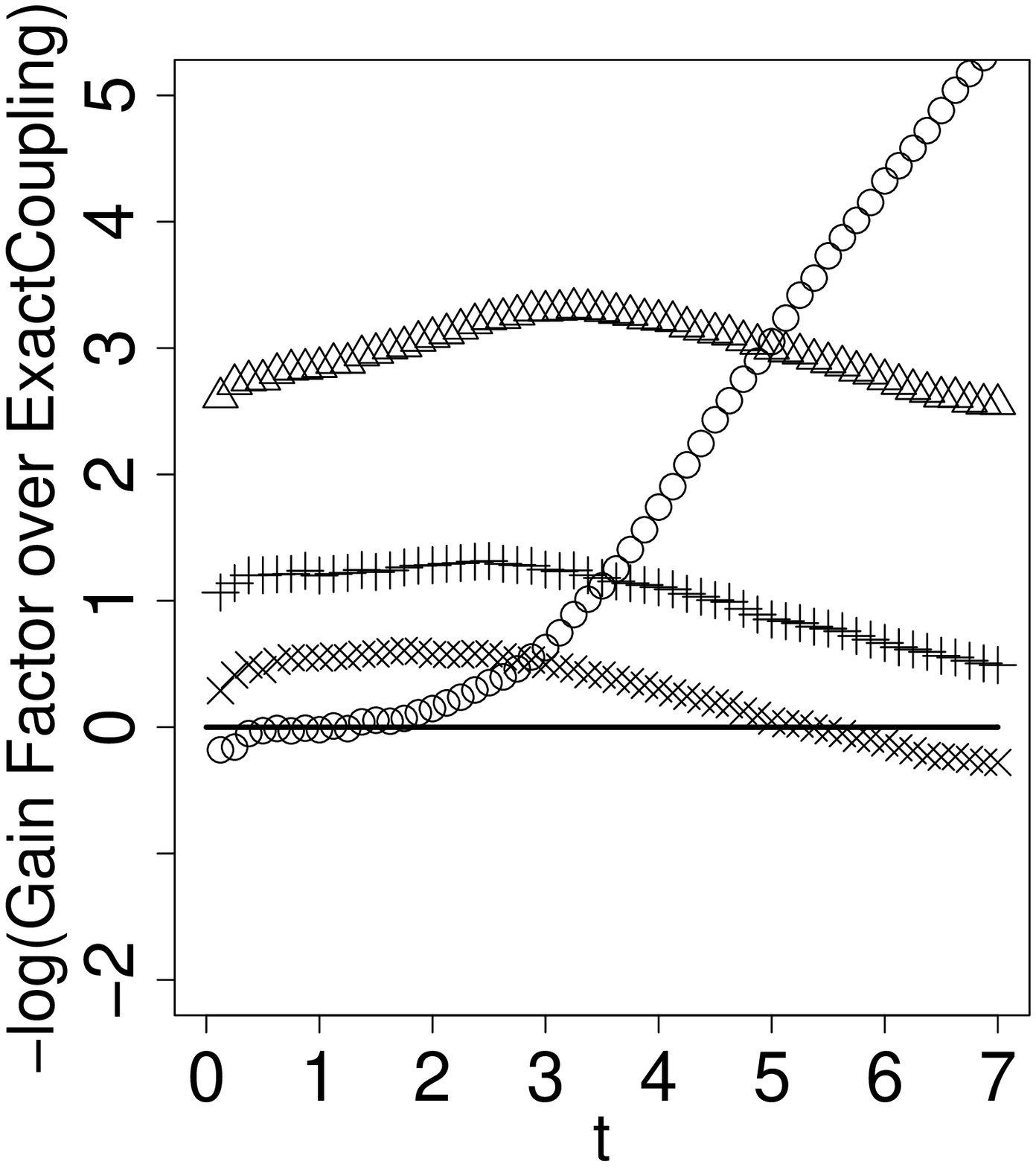}}
    }
    \subfigure[][$N=10^5$]
    {
        \resizebox{!}{0.3\linewidth}{\includegraphics{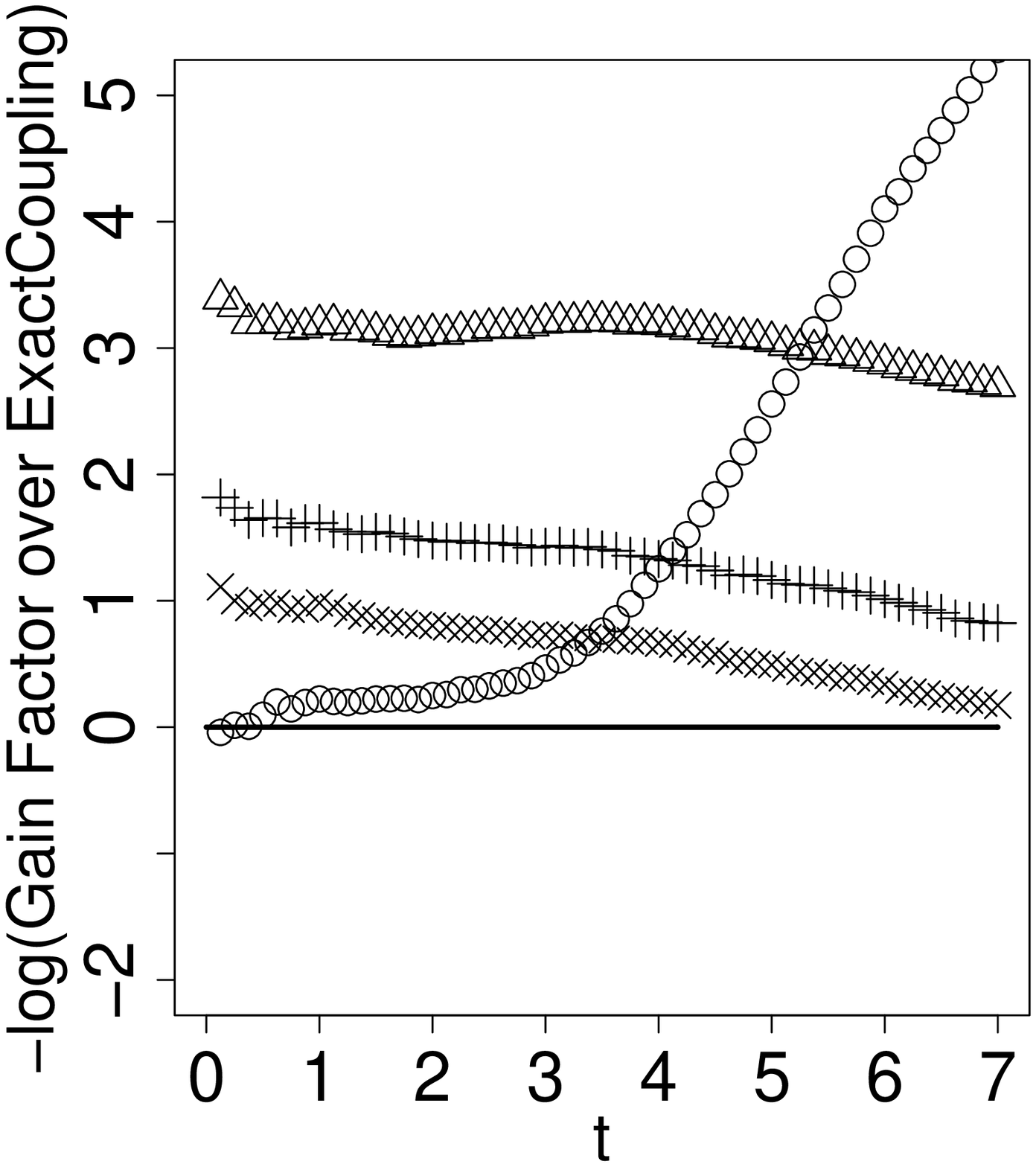}}
    }
    \caption{Additive kernel: $\log(\textrm{Gain factor}^{-1})$ as a function of $t$. The meanings of the symbols are as follows: Circles = ExactIndep, Crosses = CD($\delta\la=0.10$), Pluses = CD($\delta\la=0.05$), Triangles = CD($\delta\la=0.01$). The horizontal line is the threshold value $1.0$ for ExactCoupling.}
    \label{Add_inefficiency}
\end{figure}

\begin{figure}[!h]
    \subfigure[][$N=10^3$]
    {
        \resizebox{!}{0.3\linewidth}{\includegraphics{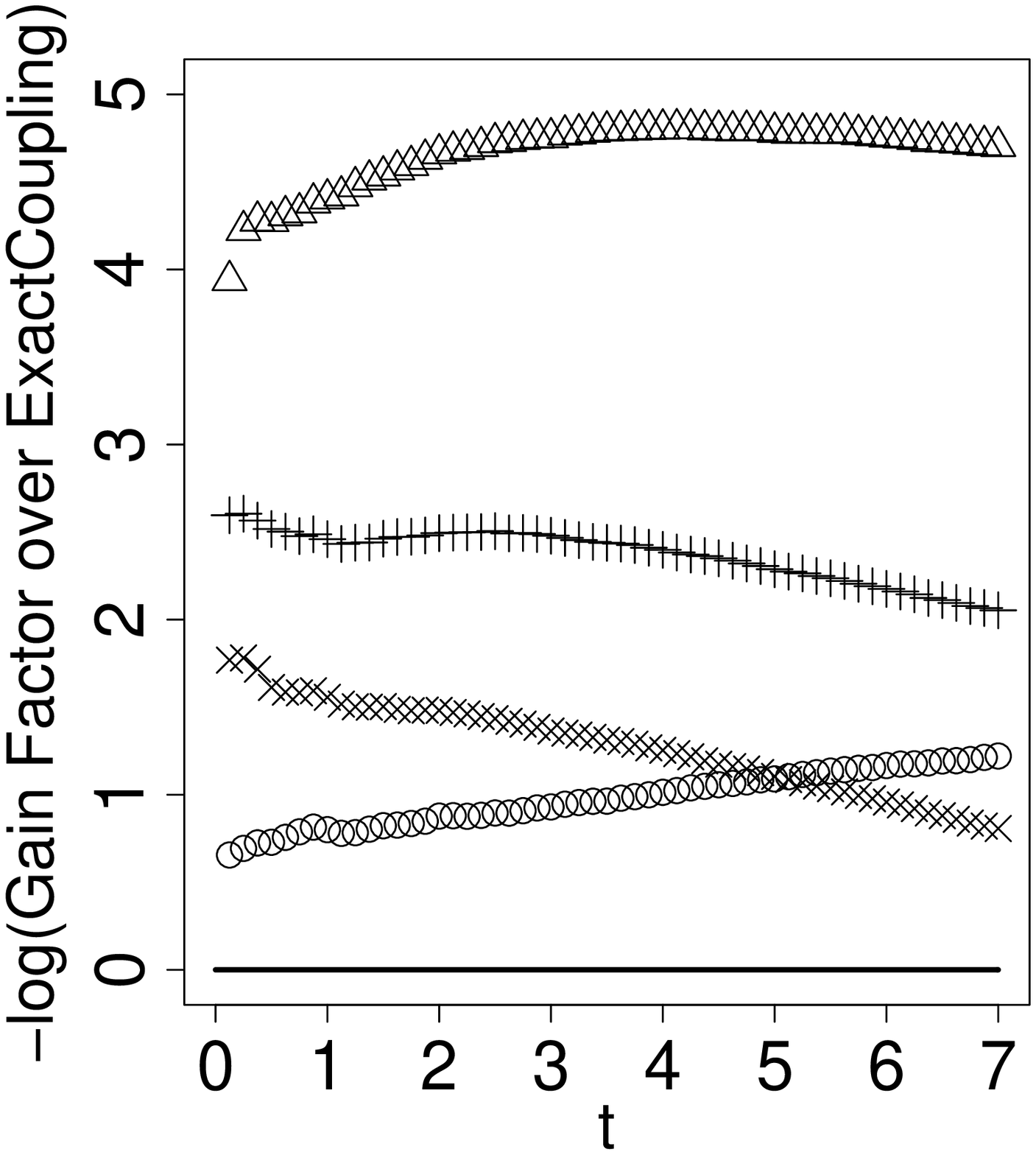}}
    }
    \subfigure[][$N=10^4$]
    {
        \resizebox{!}{0.3\linewidth}{\includegraphics{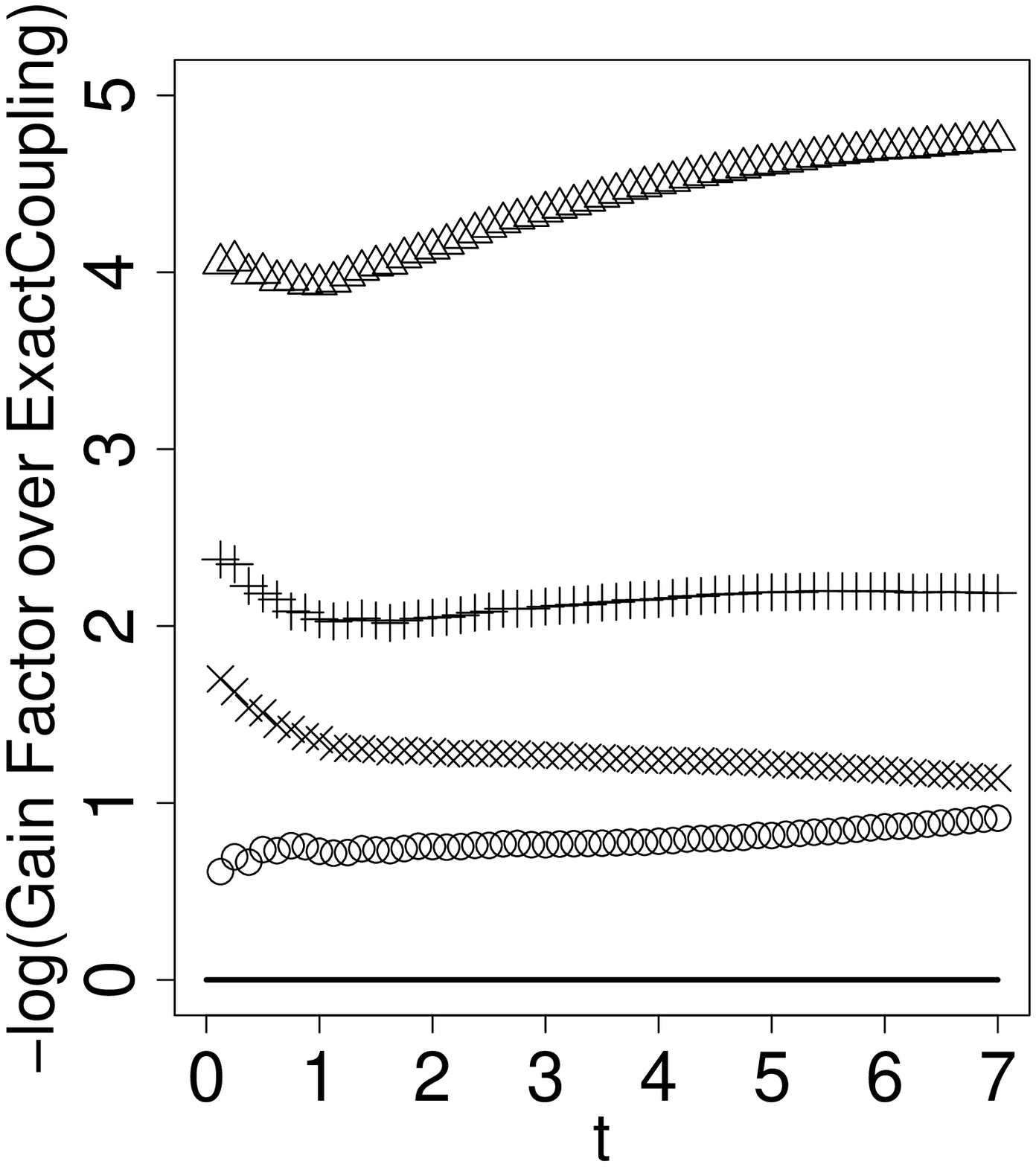}}
    }
    \subfigure[][$N=10^5$]
    {
        \resizebox{!}{0.3\linewidth}{\includegraphics{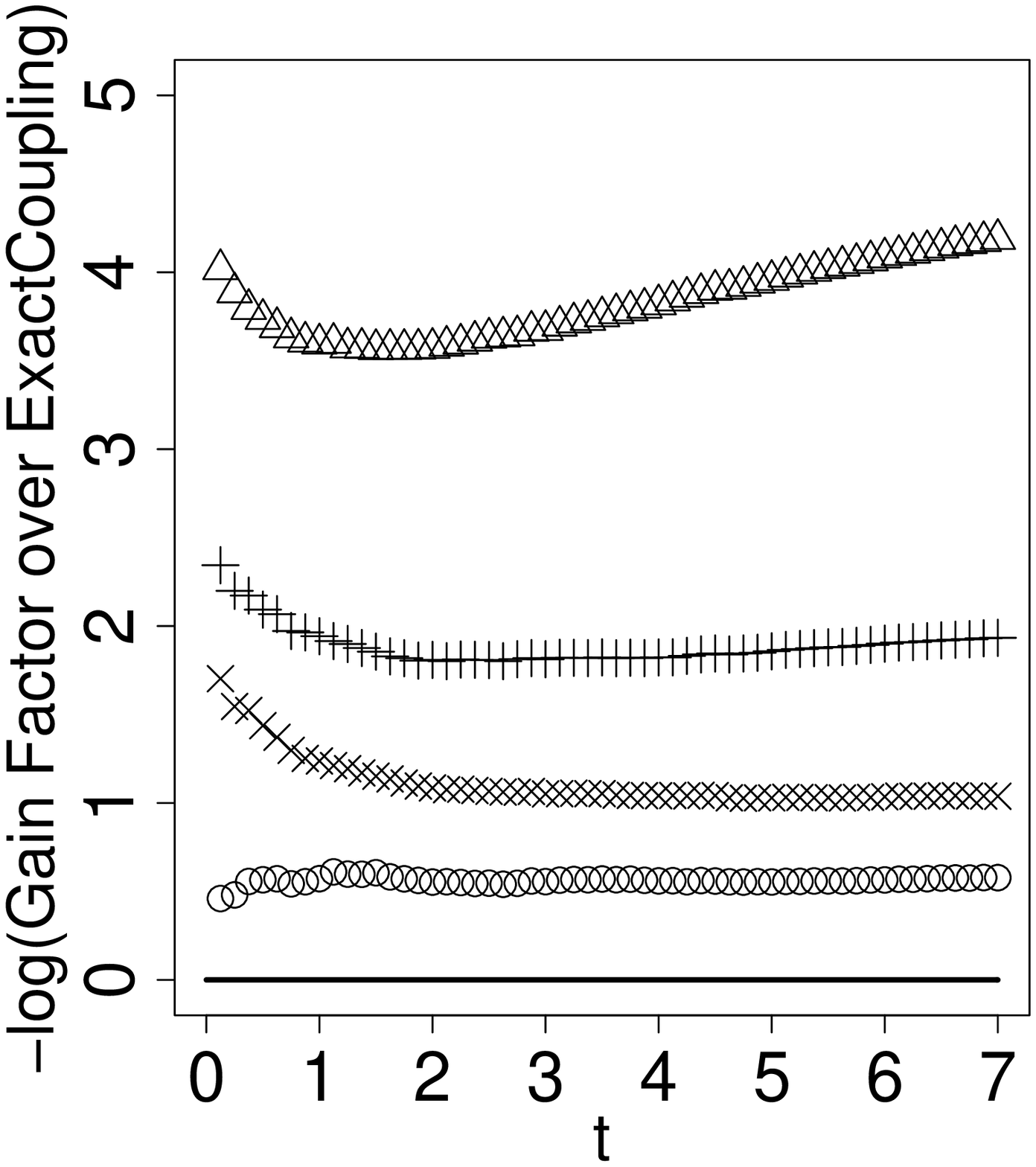}}
    }
    \caption{Soot kernel: $\log(\textrm{Gain factor}^{-1})$ as a function of $t$.  The meanings of the symbols are as follows: Circles = ExactIndep, Crosses = CD($\delta\la=0.10$), Pluses = CD($\delta\la=0.05$), Triangles = CD($\delta\la=0.01$). The horizontal line is the threshold value $1.0$ for ExactCoupling.}
    \label{Soot_inefficiency}
\end{figure}

Figures \ref{Add_inefficiency} and \ref{Soot_inefficiency} show good results. By and large, the CD algorithms appear to be considerably inferior to the ExactCoupling algorithm for the Soot kernel, and the ExactIndep algorithm in either performs slightly better than the CD ($\delta\la=0.10$). There appears to be little to moderate difference in behaviour over different values of $N$.

The picture is different for the Additive kernel. For $N=1000$, we find that the CD $(\delta\la=0.1)$ is better than the ExactCoupling, at least for very small or large times. This disadvantage gradually disappears over larger $N$ --- this is due to the increased probability of cancellations for larger $N$ which reduces the number of particles in the ensembles and therefore the CPU times. Other than this, the ExactCoupling algorithm maintains a substantial lead over the other algorithms.

\section{Conclusions}

A stochastic particle system approximation to the parametric sensitivity in Smoluchowski's coagulation equation was introduced. Rather than taking a finite difference approach to calculating sensitivities, we considered the direct parametric derivative of \eqref{SmoDiscreteSymm}, and developed a Monte-Carlo algorithm which would approximate its solution. The particle system approximation was proved to converge weakly to the solution of the sensitivity equation \eqref{SmoDiscreteSymm}, as the number of particles increases indefinitely.

The first algorithm developed (ExactIndep) allows for an exponential increase in the number of sensitivity particles. We sought to reduce this increase using several tricks: Cancellation removes `unnecessary' sensitivity particles which are needed to describe it, whilst coupling prevents their creation. These make a significant reduction to the number of particles in the ensemble. Furthermore, the resampling method puts a cap on the total number of sensitivity particles, thus stopping their exponential escalation. This gives us the ExactCoupling algorithm.

In the Numerical Results section, it was empirically confirmed that the order of convergence is $O(1/N)$ where $N$ is the number of initial particles. We then compared the Exact algorithms with those found in \cite{PeterJamesMarkus}, named here CD algorithms. It was shown that the variance of the sensitivity estimators were orders of magnitude smaller for the ExactCoupling algorithm than for the CD algorithms. However this came at the price of longer CPU run times. Two measures of efficiency, taking both the variance and the CPU time into account, were then considered. The ExactCoupling algorithm happens to require much smaller time to to reach a fixed level of error than any CD algorithm, and the gain factor, as defined in \cite{Wagner3}, also happens to be in favour of the ExactCoupling algorithm, most of the time. This definitely gives a clear advantage of our approach over finite difference methods.

However, both methods have some inherent drawback: unlike the adjoint method \cite{Vik06}, they are unidimensional in nature and compute sensitivity only for a fixed value of the parameter. It would be useful to construct a particle system approximation which do not have these weaknesses. Also, although the convergence theorem established in section \ref{SectionTheoretic} in a general framework is quite encouraging, it is not clear whether the algorithm will be as efficient as above if particles's masses can take any positive value. We leave the investigation of these questions for future work.



\end{document}